\documentclass[11pt]{article}

\usepackage{amsmath}
\usepackage{amssymb}
\usepackage{amsfonts}
\usepackage{amsthm}
\usepackage{array}
\usepackage{bm}
\usepackage{enumitem}
\usepackage{stmaryrd}

\usepackage{amsxtra}
\usepackage{array}
\usepackage{mathrsfs}
\usepackage{slashed}
\usepackage{xcolor}
\usepackage{tikz-cd}
\usepackage{tkz-euclide} 
\usepackage{pgfplots}

\newcolumntype{C}[1]{>{\centering\hspace{0pt}}p{#1}}
\usepackage[margin=1in]{geometry}

\newcommand{\GL}{\mathrm{GL}}

\newcommand{\SO}{\mathrm{SO}}
\newcommand{\Spin}{\mathrm{Spin}}
\newcommand{\U}{\mathrm{U}}
\newcommand{\SU}{\mathrm{SU}}
\newcommand{\Sp}{\mathrm{Sp}}
\newcommand{\G}{\mathrm{G}}

\newcommand{\Id}{\mathrm{Id}}
\newcommand{\Ker}{\mathrm{Ker}}

\newcommand{\End}{\mathrm{End}}

\newcommand{\Sym}{\mathrm{Sym}}

\newcommand{\Z}{\mathbb{Z}}

\newcommand{\R}{\mathbb{R}}
\newcommand{\C}{\mathbb{C}}
\newcommand{\HH}{\mathbb{H}}

\newcommand{\RP}{\mathbb{RP}}
\newcommand{\CP}{\mathbb{CP}}
\newcommand{\HP}{\mathbb{HP}}

\newcommand{\Gr}{\mathrm{Gr}}
\newcommand{\vol}{\mathrm{vol}}
\newcommand{\Hol}{\mathrm{Hol}}

\newtheorem{thm}{Theorem}[section]
\newtheorem{prop}[thm]{Proposition}
\newtheorem{lem}[thm]{Lemma}
\newtheorem{cor}[thm]{Corollary}

\theoremstyle{definition}
\newtheorem{defn}[thm]{Definition}
\newtheorem{example}[thm]{Example}
\newtheorem{rmk}[thm]{Remark}

\usepackage[nottoc]{tocbibind}
\numberwithin{equation}{section}

\usepackage[hidelinks]{hyperref}

\title{Associative Submanifolds of \\ Squashed $3$-Sasakian Manifolds}
\author{Gavin Ball and Jesse Madnick}
\date{September 9, 2022}

\newcommand{\Addresses}
{{  \bigskip
		\textsc{University of Wisconsin-Madison} \par\nopagebreak
		\textsc{Department of Mathematics}\par\nopagebreak
		\textsc{Madison, WI, USA}\par\nopagebreak
		\texttt{gball3@wisc.edu} \\
		
		\bigskip
		\textsc{National Center for Theoretical Sciences} \par\nopagebreak
		\textsc{National Taiwan University}\par\nopagebreak
		\textsc{Taipei, Taiwan}\par\nopagebreak
		\texttt{jmadnick@ncts.ntu.edu.tw}
}}

\begin{document}

\maketitle

\begin{abstract}
Every compact $3$-Sasakian $7$-manifold $M$ admits a canonical $2$-parameter family of co-closed $\G_2$-structures $\varphi_{a,b}$ for $a,b > 0$, as well as a foliation by $\varphi_{a,b}$-associative $3$-folds whose leaf space $X$ is a positive quaternion-K\"{a}hler $4$-orbifold.  We prove that associative $3$-folds in $(M,\varphi_{a,b})$ that are ruled by a certain type of geodesic are in correspondence with pseudo-holomorphic curves in the almost-complex $8$-manifold $Z \times S^2$, where $Z$ is the twistor space of $X$ equipped with its strict nearly-K\"{a}hler structure. \\
\indent As an application, we construct infinitely many topological types of non-trivial, compact associative $3$-folds in the squashed $7$-spheres $(S^7, \varphi_{a,b})$ and squashed exceptional Aloff-Wallach spaces $(N_{1,1}, \varphi_{a,b})$.  Topologically, our examples are circle bundles over a genus $g$ surface, for any $g \geq 0$.
\end{abstract}

\tableofcontents


\section{Introduction}

\indent \indent Let $(M,\varphi)$ be a $7$-manifold with a $\G_2$-structure $\varphi$.  An \emph{associative $3$-fold} is a $3$-dimensional submanifold $N \subset M$ for which
$$\varphi|_N = \vol_N.$$
Associative submanifolds are conjectured to play an important role in defining an enumerative invariant of torsion-free $\G_2$-structures counting $\G_2$-instantons \cite{DonTho, DonSeg, walpuski2017g2, Joyce18Count, Hay19}. In this direction, it is desirable to have a large supply of associative submanifolds, especially examples that exist in continuous families of coclosed $\G_2$-structures $\varphi_t$. In general, associative $3$-folds are difficult to construct, making conjectures hard to test.  In this work, we construct examples of compact $3$-folds in squashed 3-Sasakian 7-manifolds $M$ that are associative with respect to a $2$-parameter family of $\G_2$-structures. \\
\indent A second motivation for this work comes from $\Spin(7)$ geometry.  Among the co-closed $\G_2$-structures, there is a distinguished class of \emph{nearly-parallel $\G_2$-structures}, defined by the property that their metric cone is a $\Spin(7)$-manifold.  Associative submanifolds in nearly-parallel $7$-manifolds $M$ provide models for conically singular Cayley manifolds in the cone $(\mathrm{C}(M), g_{\mathrm{C}}) = (\R^+ \times M, dr^2 + r^2g)$. \\
\indent Examples of associative submanifolds in nearly-parallel $7$-manifolds have been constructed by Fox \cite{fox2007cayley}, Lotay \cite{lotay2012associative}, Kawai \cite{kawai2015some}, and the authors \cite{BMBerger}.  Gauge theory on nearly-parallel $7$-manifolds has been studied by Ball and Oliveira \cite{BaOlAloffWallach}, Waldron \cite{waldron2020mathrm}, Lotay and Oliveira \cite{LoOl22}, and Singhal \cite{Singhal22}.  \\

\indent A Riemannian $7$-manifold $(M, g)$ is \emph{$3$-Sasakian} if its metric cone $(\text{C}(M), g_{\mathrm{C}})$ is hyperk\"{a}hler.  Compact examples include the $7$-sphere $S^7$, the Aloff-Wallach space $N_{1,1} = \SU(3)/S^1_{(1,1)}$, and the $7$-manifolds $S(p_1, p_2, p_3)$ constructed by Boyer-Galicki-Mann \cite{mann1994geometry}.  Since $\mathrm{C}(M)$ is hyperk\"{a}hler, it admits a triple $(I_1, I_2, I_3)$ of $g_{\mathrm{C}}$-orthogonal complex structures satisfying $I_1I_2 = I_3$ for which the associated $2$-forms $(\omega_1, \omega_2, \omega_3)$ are closed.  Therefore, viewing $M$ as the link $\{1\} \times M \subset \mathrm{C}(M)$, we see that $TM$ splits as
$$TM = \mathsf{A} \oplus \mathsf{C}, \ \ \mathsf{A} := \text{span}(I_1(\partial_r), I_2(\partial_r), I_3(\partial_r)), \ \ \mathsf{C} := \mathsf{A}^\perp.$$
On $M$, we define $1$-forms $\alpha_j := \omega_j(\partial_r, \cdot)$ and $2$-forms $\Omega_j := \omega_j|_{\mathsf{C}}$.
These yield a canonical $2$-parameter family of co-closed $\G_2$-structures on $M$ via
\begin{equation*}
    \varphi_{a,b} := a^3\,\alpha_{123} - ab^2\sum \alpha_j \wedge \Omega_j
\end{equation*}
whose induced Riemannian metrics are
$$g_{a,b} := \left.a^2 g\right|_{\mathsf{A}}+ \left.b^2g\right|_{\mathsf{C}}.$$
That $\varphi_{a,b}$ is co-closed follows from calculations in \cite[$\S$5]{friedrich1997nearly}.  See also \cite{agricola20103}. \\
\indent We refer to $(M, \varphi_{a,b})$ as a \emph{squashed $3$-Sasakian $7$-manifold}.  The family of co-closed $\G_2$-structures $\varphi_{a,b}$ exhibits some interesting features.  For example, $\varphi_{a,b}$ is nearly-parallel if and only if $b^2 = 5a^2$, and the metric $g_{1/\sqrt{5},1}$ is properly nearly-parallel.  Moreover, while $\varphi_{1,1}$ is not nearly-parallel, it is nevertheless \emph{isometric} to a nearly-parallel $\G_2$-structure $\widehat{\varphi}$ as its induced metric $g = g_{1,1}$ is $3$-Sasakian.  See $\S$\ref{sec:Squashed3Sas} for details. \\

\indent In this work, we study associative $3$-folds in $(M, \varphi_{a,b})$, where $M$ is compact.  The simplest examples are the \emph{canonical leaves}, i.e., those associatives that are tangent to $\mathsf{A}$.  Indeed, the distribution $\mathsf{A} \subset TM$ is integrable, yielding a \emph{canonical foliation} $\mathcal{F}_A$ of $M$.  The canonical leaves turn out to be totally-geodesic, $\varphi_{a,b}$-associative $3$-folds in $M$ with constant positive sectional curvature.  It is known that the leaf space $X := M/\mathcal{F}_A$ is a positive quaternion-K\"{a}hler $4$-orbifold, and that the projection $h \colon M \to X$ is a principal $\SU(2)$- or $\SO(3)$-orbibundle \cite[$\S$13.3.13]{boyer2007sasakian}.  For example, when $M = S^7$, the leaf space $X = S^4$, the map $h \colon S^7 \to S^4$ is the quaternionic Hopf fibration, and the canonical leaves are the Hopf $3$-spheres. \\
\indent To obtain new examples, we consider a more general geometric ansatz, studying those associatives in $(M,\varphi_{a,b})$ that are ruled by \emph{Hopf circles}, those closed geodesics that lie in a canonical leaf.  In other words, a Hopf circle is an orbit of the locally free $S^1$-action on $M$ generated by the the Reeb field $A_w := w_1 I_1(\partial_r) + w_2I_2(\partial_r) + w_3I_3(\partial_r)$, for any $w = (w_1, w_2, w_3) \in S^2$.  Each quotient space $M/\langle A_w\rangle$ is diffeomorphic to the twistor space $Z$ of $X$, and the projection $p_w \colon M \to Z$ is a principal $S^1$-orbibundle \cite[$\S$7.5.1, $\S$13.3.1]{boyer2007sasakian}.  For example, when $M = S^7$, we have $Z = \CP^3$, and each $p_w \colon S^7 \to \CP^3$ is a complex Hopf fibration. \\
\indent The space of Hopf circles in $M$ is diffeomorphic to $Z \times S^2$, so any Hopf-ruled $3$-dimensional submanifold of $M$ (not necessarily associative) can be viewed as a surface in $Z \times S^2$.  Conversely, any surface $\Sigma \subset Z \times S^2$ gives a corresponding Hopf-ruled $3$-fold $\Gamma(\Sigma) \subset M$.  We then ask which surfaces in $Z \times S^2$ yield \emph{associative} Hopf-ruled $3$-folds in $M$.  The Correspondence Theorem \ref{thm:MainCorrespondence} provides the answer: it is the $J$-holomorphic curves in $Z \times S^2$.  Here, $J$ is the product almost-complex structure on $Z \times S^2$, where $Z$ carries its standard strict nearly-K\"{a}hler structure. \\
\indent Recalling that that $Z \to X$ is an $S^2$-bundle, we view $Z \times S^2 \to X$ as an $(S^2 \times S^2)$-bundle in the obvious way.  We can now state the result:

\begin{thm}[Correspondence Theorem] \label{thm:MainCorrespondence} Let $(M^7, \varphi_{a,b})$ be a compact squashed $3$-Sasakian $7$-manifold, where $a,b > 0$ are fixed.  Let $Z$ be the twistor space of $X := M/\mathcal{F}_A$.
\begin{enumerate}[label=(\alph*)]
\item Every Hopf-ruled associative $3$-fold in $(M^7, \varphi_{a,b})$ is locally of the form $\Gamma(\Sigma)$ for some $J$-holomorphic curve $\Sigma \subset Z \times S^2$.
\item Let $\Sigma \subset Z \times S^2$ be a $J$-holomorphic curve.
\begin{enumerate}[label=\arabic*.]
     \item If $\Sigma$ does not lie in an $(S^2 \times S^2)$-fiber, then there exists a discrete set $D \subset \Sigma$ such that $\Gamma(\Sigma - D) \subset M$ is a Hopf-ruled associative $3$-fold. 
    \item If $\Sigma$ lies in an $(S^2 \times S^2)$-fiber, then $\Gamma(\Sigma)$ is a subset of a canonical leaf.
\end{enumerate}
\end{enumerate}
\end{thm}

\indent The proof of Theorem \ref{thm:MainCorrespondence}(b) gives more than the statement.  Indeed, let $\Sigma_g$ be a compact surface of genus $g \geq 0$.  In $\S$\ref{sec:Correspondence}, we observe that if $f \colon \Sigma_g \to Z$ is an immersed pseudo-holomorphic curve in $Z$ that is horizontal for the twistor fibration $\tau \colon Z \to X$, then for any holomorphic map $w \colon \Sigma_g \to S^2$, the surface $(f,w) \colon \Sigma_g \to Z \times S^2$ is a $J$-holomorphic curve that does not lie in an $(S^2 \times S^2)$-fiber and has $D = \O$.  Consequently, $\Gamma(f,w) \subset M$ is a compact associative $3$-fold in $M$ diffeomorphic to an $S^1$-bundle over $\Sigma_g$. \\

\indent As an application, we consider the case of $M = S^7$ equipped with any of its $\Sp(2)\Sp(1)$-invariant $\G_2$-structures $\varphi_{a,b}$.  For each $g \geq 0$, Bryant \cite{bryant1982conformal} has constructed an embedded horizontal pseudo-holomorphic curve $f \colon \Sigma_g \to \CP^3 = Z$.  Taking $w \colon \Sigma_g \to S^2$ to be a constant map $w = w_0$ yields the \emph{trivial} associative $3$-fold $\Gamma(f,w_0) = p_{w_0}^{-1}(f(\Sigma_g)) \subset S^7$.  On the other hand, taking $w$ to be a non-constant holomorphic map yields the first non-trivial compact $\varphi_{a,b}$-associative $3$-folds in $S^7$.  That is:

\begin{thm} \label{thm:AssocSphere} Fix $a,b > 0$.  For every $g \geq 0$, there exists a non-trivial compact associative $3$-fold in $(S^7, \varphi_{a,b})$ diffeomorphic to an $S^1$-bundle over a genus $g$ surface.
\end{thm}

\noindent Note that non-trivial compact $\widehat{\varphi}$-associative $3$-folds in $S^7$ have been constructed in \cite{fox2007cayley} and \cite{lotay2012associative}, where $\widehat{\varphi}$ is the $\Spin(7)$-invariant $\G_2$-structure.  See $\S$\ref{sec:7Sphere} for a more detailed discussion. \\

\indent Similarly, we may apply our theory to $M = N_{1,1}$, in which case $Z = \SU(3)/\mathrm{T}^2$.  In this case, horizontal pseudo-holomorphic curves in $\SU(3)/\mathrm{T}^2$ may be constructed as lifts of holomorphic curves in $\mathbb{CP}^2.$ Every compact Riemann surface $\Sigma$ can be realized as an immersed holomorphic curve in $\mathbb{CP}^2.$ Taking $w: \Sigma \to S^2$ to be a non-constant holomorphic map, we obtain:

\begin{thm} \label{thm:AssocAW} Fix $a,b > 0$.  For every $g \geq 0$, there exists a non-trivial compact associative $3$-fold in $(N_{1,1}, \varphi_{a,b})$ diffeomorphic to an $S^1$-bundle over a genus $g$ surface.
\end{thm}

\indent In sum, we can obtain non-trivial compact $\varphi_{a,b}$-associative $3$-folds in a squashed $3$-Sasakian $7$-manifold $M$ by constructing horizontal pseudo-holomorphic curves in the corresponding nearly-K\"{a}hler orbifold $Z$.  For this reason, it would be desirable to have a large supply of such pseudo-holomorphic curves in twistor spaces $Z$ other than $\CP^3$ and $\SU(3)/\mathrm{T}^2$.  Such twistor spaces are orbifolds rather than manifolds, and it seems likely that techniques from complex algebraic geometry would prove useful in attacking this problem.

\indent Note that Theorem \ref{thm:MainCorrespondence}(b) applies to all values of the squashing parameters $a,b > 0$.  That is, for any $J$-holomorphic curve $\Sigma \subset Z \times S^2$, there are corresponding Hopf-ruled $\varphi_{a,b}$-associative $3$-folds $\Gamma(\Sigma - D) \subset M$ for every $a,b > 0$.  This raises the following question: For a fixed $J$-holomorphic curve $\Sigma \subset Z \times S^2$, to what extent does the deformation theory of the $\varphi_{a,b}$-associative $3$-fold $\Gamma(\Sigma - D) \subset M$ depend on the squashing parameters?  This is an interesting direction for further investigation. \\

\indent Finally, we comment briefly on the organization of this work.  Although $\S$\ref{sec:Prelim} consists of preliminary material, $\S$\ref{sec:SO4-Structures} contains several novelties.  In particular, Proposition \ref{prop:assocorbits} does not seem to be widely known.  The latter motivates the concept of ``striped” associative $3$-folds (see Definition \ref{def:striped}), which likely has applications beyond this work. \\ 
\indent In $\S$\ref{sec:Squashed3Sas}, we study squashed 3-Sasakian $7$-manifolds $(M,\varphi_{a,b})$ via the method of moving frames.  In $\S$\ref{sec:HopfCircle}, we show that the space of Hopf circles is diffeomorphic to $Z \times S^2$, and we equip this space with a canonical $2$-parameter family of almost-Hermitian structures $(h_{a,b}, J, \omega_{a,b})$.  Our study of associative $3$-folds in $(M, \varphi_{a,b})$ begins in $\S$\ref{sec:AssocSquash3-Sas}. \\

\noindent \textbf{Acknowledgements:} We thank Benjamin Aslan and Benoit Charbonneau for clarifying comments.  The second author thanks McKenzie Wang and Chung-Jun Tsai for conversations, and thanks McMaster University and the National Center for Theoretical Sciences for their support.

\subsection{Discussion: A Geometric View of the Correspondence Theorem}

\indent \indent We wish to explain in heuristic terms the geometric idea behind the Correspondence Theorem \ref{thm:MainCorrespondence} in the case of the $7$-sphere equipped with one of the co-closed $\G_2$-structures $\varphi_{a,b}$.  Viewing $M = S^7 \subset \HH^2$ and $X = S^4 \cong \HP^1$, note that each fiber of the quaternionic Hopf fibration $h \colon S^7 \to S^4$ is a $3$-sphere given by the intersection of $S^7$ with a quaternionic line $H$.  There is an $S^2$-family of identifications $w \colon H \xrightarrow{\cong} \C^2$, each of which yields a complex Hopf fibration $p_w \colon S^7 \to \CP^3$.  A Hopf circle in $S^7$ is an $S^1$-fiber $p_w^{-1}(z)$ for some fixed $w \in S^2$ and $z \in \CP^3$. \\
\indent Now, for every immersed pseudo-holomorphic curve $f \colon \Sigma \to \CP^3$ and every fixed $w \in S^2$, the pre-image $p_w^{-1}(f(\Sigma)) \subset S^7$ is a trivial associative $3$-fold in $S^7$, where ``pseudo-holomorphic" is meant with respect to the strict nearly-K\"{a}hler structure on $\CP^3$.  A natural question, then, is whether one can generalize the class of trivial associatives by allowing the ruling direction $w \in S^2$ to vary with $z \in \Sigma$ in some fashion.  In other words, which pairs of ``directrix surface" $f \colon \Sigma \to \CP^3$ and ``ruling function" $w \colon \Sigma \to S^2$ result in an immersed Hopf-ruled $3$-fold
\begin{equation} \label{eq:RuledSet}
    \bigcup_{z \in \Sigma} p^{-1}_{w(z)}(f(\Sigma)) \subset S^7
\end{equation}
that is associative?  Theorem \ref{thm:MainCorrespondence} answers this by relating $\varphi_{a,b}$-associatives in $S^7$ of the form (\ref{eq:RuledSet}) to pairs $(f,w) \colon \Sigma \to \CP^3 \times S^2$ consisting of a pseudo-holomorphic map $f$ and holomorphic map $w$ for which $\text{rank}(f,w) = 2$. \\
\indent The technical point is that certain pairs of directrix surface $f$ and ruling function $w$ yield Hopf-ruled $3$-folds $\Gamma(f,w) \subset S^7$ that are not immersed.  This sort of issue is endemic to the theory of ruled submanifolds, occurring even in the classical case of ruled surfaces in $\R^3$.  Indeed, while one would like to say that any function $\Gamma \colon \R^2 \to \R^3$ of the form $\Gamma(s,t) := \mathbf{f}(s) + t\mathbf{w}(s)$ with $\mathbf{f}' \cdot \mathbf{w} = 0$ and $\Vert \mathbf{w} \Vert = 1$ parametrizes a ruled surface, certain pairs of ($\mathbf{f}, \mathbf{w}$) result in degenerations.  For example, taking $\mathbf{f}(s) = \mathbf{0}$ and $\mathbf{w}(s) = \frac{1}{\sqrt{2}}(\cos(s), \sin(s), 1)$, the resulting map $\Gamma$ parameterizes a cone in $\R^3$, but is not an immersion on $\{(s,t) \colon t = 0\}$.  In our situation, assuming that $f$ is pseudo-holomorphic, that $w$ is holomorphic, and that the image of $(f,w)$ doesn't lie in an $(S^2 \times S^2)$-fiber, Theorem \ref{thm:MainCorrespondence}(b) shows that one only need remove a discrete set $D$ from $\Sigma$ to obtain an immersed Hopf-ruled associative $\Gamma(\Sigma - D) \subset S^7$.

\subsection{Terminology}

\indent \indent (1) Given a Riemannian $n$-manifold $(M, g)$, the \emph{(metric) cone over $(M,g)$} is the Riemannian $(n+1)$-manifold
$$(\mathrm{C}(M), g_{\mathrm{C}}) := (\R^+ \times M, dr^2 + r^2g).$$
We always identify $M$ with the \emph{link} $\{1\} \times M \subset \text{C}(M)$. \\

\indent (2) On a $7$-manifold $M$, an \emph{$\SO(4)$-structure} is a pair $(\varphi, \mathsf{A})$ consisting of a $\G_2$-structure $\varphi \in \Omega^3(M)$ and a distribution $\mathsf{A} \subset TM$ of associative $3$-planes.  The following two related concepts are fundamental to this work, both of which are discussed in $\S$\ref{sec:SO4-Structures}.
\begin{itemize}
    \item When $M$ carries an $\SO(4)$-structure $(\varphi, \mathsf{A})$, an associative $3$-fold $N \subset M$ is called \emph{striped} if $\dim(T_xN \cap \mathsf{A}_x) = 1$ for each $x \in N$.  
    \item An $\SO(4)$-structure $(\varphi, \mathsf{A})$ on $M$ is called \emph{quaternion-Sasakian} if the induced $\Sp(2)\Sp(1)$-structure on $\mathrm{C}(M)$ is quaternion-K\"{a}hler.  Quaternion-Sasakian $7$-manifolds are virtually the same as $3$-Sasakian $7$-manifolds, but entail slightly less structure.  See $\S$\ref{sec:SO4-Structures} for details.
\end{itemize}

\section{Preliminaries} \label{sec:Prelim}

\subsection{$\G_2$-Structures and Associative $3$-folds}

\indent \indent Let $M^7$ be a smooth $7$-manifold.  A \emph{$\G_2$-structure} on $M$ is a $3$-form $\varphi \in \Omega^3(M)$ such that at each $x \in M$, the symmetric bilinear form $B_\varphi|_x \colon \Sym^2(T^*_xM) \to \Lambda^7(T^*_xM)$ given by $B_\varphi(u,v) := \textstyle \frac{1}{6}(\iota_u\varphi) \wedge (\iota_v\varphi) \wedge \varphi$ is definite. This is equivalent to the condition that at each $x \in M$, there exists a frame $(e_1, \ldots, e_7)$ at $x$ for which
$$\varphi|_x = e^{123} + e^{145} + e^{167} + e^{246} - e^{257} - e^{347} - e^{356}$$
where $(e^1, \ldots, e^7)$ is the dual coframe and $e^{ijk} := e^i \wedge e^j \wedge e^k$.  It is well-known that a $7$-manifold admits a $\G_2$-structure if and only if it is orientable and spin.  A $\G_2$-structure $\varphi \in \Omega^3(M)$ naturally induces a Riemannian metric $g_\varphi$ and orientation $\vol_\varphi \in \Omega^7(M)$ via
\begin{align*}
\textstyle g_\varphi(u,v)\,\vol_\varphi & = \textstyle \frac{1}{6}(\iota_u\varphi) \wedge (\iota_v\varphi) \wedge \varphi
\end{align*}
Two $\G_2$-structures $\varphi$, $\widehat{\varphi}$ on $M$ are called \emph{isometric} if $g_\varphi = g_{\widehat{\varphi}}$. \\

\indent A $\G_2$-structure $\varphi \in \Omega^3(M)$ is said to be \textit{co-closed} if $d\ast\varphi = 0$.  All of the $\G_2$-structures considered in this work will be co-closed.  If $\varphi$ is co-closed, then it is well-known that
$$d\varphi = \tau_0\ast\!\varphi + \ast\tau_3$$
for a unique function $\tau_0 \in \Omega^0(M)$ and $3$-form $\tau_3 \in \Omega^3_{27}(M) = \Gamma(\Lambda^3_{27}(M))$, where here
$$\Lambda^3_{27}(M) := \left\{ \gamma \in \Lambda^3(T^*M) \colon \gamma \wedge \varphi = 0 \text{ and }  \gamma \wedge \ast\varphi = 0 \right\}\!.$$
A $\G_2$-structure $\varphi$ is called \textit{nearly-parallel} if $d\varphi = \lambda \ast\! \varphi$ for some for some non-zero constant $\lambda \in \R$.  It is clear that every nearly-parallel $\G_2$-structure is co-closed.  If $\varphi$ is a nearly-parallel $\G_2$-structure, then the induced metric $g_\varphi$ is Einstein with $\text{Scal}(g_\varphi) = \frac{21}{8}\lambda^2 > 0$. \\

\indent An associative $3$-plane is an oriented $3$-dimensional subspace $E \subset T_xM$ such that $\varphi|_E = \vol_E$.  An \emph{associative $3$-fold} is an oriented submanifold $\Sigma^3 \subset (M^7, \varphi)$ for which each tangent space $T_x\Sigma \subset T_xM$ is an associative $3$-plane --- i.e., a submanifold satisfying $\varphi|_\Sigma = \pm\vol_\Sigma$.  Associative $3$-folds are fundamental sub-objects in $\G_2$-geometry, and are formally analogous to holomorphic curves in Calabi-Yau $3$-folds.

\begin{rmk} Let $(M^7, \varphi)$ be a $7$-manifold with a co-closed $\G_2$-structure $\varphi$.  It is natural to ask under what conditions an associative $3$-fold $\Sigma \subset M$ is a minimal submanifold.  To answer this, consider the $\G_2$-vector bundle $TM|_\Sigma$.  The splitting $TM|_\Sigma = T\Sigma \oplus N\Sigma$ reduces the structure group to $\SO(4) \leq \G_2$, and so each $\Lambda^k(T^*M|_\Sigma)$ is an $\SO(4)$-vector bundle as well.  In particular, is shown in \cite{ball2020meanExcept} that $\left.\Lambda^3_{27}(M)\right|_{\Sigma}$ decomposes into $\SO(4)$-irreducible vector subbundles
$$\left.\Lambda^3_{27}(M)\right|_\Sigma = (\Lambda^3_{27})_{0,0} \oplus (\Lambda^3_{27})_{0,4} \oplus (\Lambda^3_{27})_{1,1}  \oplus (\Lambda^3_{27})_{1,3} \oplus (\Lambda^3_{27})_{2,2}$$
of ranks $1, 5, 4, 8$, and $9$, respectively.  Moreover, there is a natural $\SO(4)$-vector bundle isomorphism $\natural \colon N\Sigma \to (\Lambda^3_{27})_{1,1}$.  In this language, the mean curvature vector field $H \in \Gamma(N\Sigma)$ of the associative $\Sigma \subset M$ is given by
$$H^\natural = \textstyle -\frac{\sqrt{3}}{2}(\tau_3)_{1,1}$$
Thus, the generic associative $3$-fold in $(M, \varphi)$ is not minimal.  However, if the co-closed $\G_2$-structure satisfies $\tau_3 = 0$ (e.g., if it is nearly-parallel or torsion-free), then every associative in $M$ is minimal.
\end{rmk}

\subsection{$\SO(4)$-Structures and Associative $3$-folds} \label{sec:SO4-Structures}

\indent \indent The $7$-manifolds $M$ that we consider in this work will carry an additional bit of data beyond the co-closed $\G_2$-structure.  Namely:

\begin{defn} Let $M^7$ be a smooth $7$-manifold.  An \emph{$\SO(4)$-structure} on $M$ is a pair $(\varphi, \mathsf{A})$ consisting of a $\G_2$-structure $\varphi \in \Omega^3(M)$ and an associative $3$-plane distribution $\mathsf{A} \subset TM$.  The geometry of $\SO(4)$-structures has been studied in \cite[$\S$7 -- $\S$11]{ball2019seven}. 
\end{defn}

\indent If $M$ admits an $\SO(4)$-structure $(\varphi, \mathsf{A})$, then its metric cone $(\mathrm{C}(M), g_{\mathrm{C}})$ admits a compatible $\Spin^h(4)$-structure, where $\Spin^h(4) = (\SU(2) \times \SU(2) \times \SU(2))/\Z_2$.  That is, $\mathrm{C}(M)$ admits a conical $\Spin(7)$-structure $\Phi \in \Omega^4(\mathrm{C}(M))$ together with a preferred Cayley $4$-plane distribution $\mathsf{K} \subset T(\mathsf{C}(M))$ via
\begin{align*}
    \Phi & = r^3dr \wedge \varphi + r^4\ast\! \varphi & \mathsf{K} = \mathsf{A} \oplus \text{span}(\partial_r).
\end{align*}
More importantly for this work, since $\Spin^h(4) \leq \Sp(2)\Sp(1)$, the $8$-manifold $\text{C}(M)$ also admits a conical $\Sp(2)\Sp(1)$-structure.  To be explicit, note that the volume form of $M$ decomposes as $\vol = \vol_{\mathsf{A}} \wedge \ast\vol_{\mathsf{A}}$, where $\vol_{\mathsf{A}} \in \Omega^3(M)$ denotes the volume form of $\mathsf{A}$ and all Hodge stars are taken on $M$.  Then the $4$-form on $\Theta \in \Omega^4(\mathsf{C}(M))$ given by
\begin{align*}
    \Theta & = r^3dr \wedge \left( 4\vol_{\mathsf{A}} - \varphi \right) + r^4\! \ast\!(4\vol_{\mathsf{A}} - \varphi)
\end{align*}
is a conical $\Sp(2)\Sp(1)$-structure. \\
\indent Interestingly, one does not need the full $\Spin^h(4)$-structure on $\mathrm{C}(M)$ to recover the $\SO(4)$-structure on $M$.  Indeed, if $\mathrm{C}(M)$ carries merely a conical $\Sp(2)\Sp(1)$-structure $\Theta \in \Omega^4(\text{C}(M))$, then $M$ inherits an $\SO(4)$-structure $(\varphi, \mathsf{A})$ via
\begin{align*}
    \varphi & = 4\vol_{\mathsf{A}} - \ast_M\Theta|_M & \mathsf{A} = \text{span}( I_1(\partial_r), I_2(\partial_r), I_3(\partial_r)),
\end{align*}
where $(I_1, I_2, I_3)$ are locally-defined almost-complex structures on $\mathrm{C}(M)$ satisfying the quaternionic relations ($I_1I_2 = I_3$, etc.) that arise from the $\Sp(2)\Sp(1)$-structure.   This motivates the following:

\begin{defn} An $\SO(4)$-structure $(\varphi, \mathsf{A})$ on $M$ is called \emph{quaternion-Sasakian} if the induced $\Sp(2)\Sp(1)$-structure on $\mathrm{C}(M)$ is quaternion-K\"{a}hler.  
\end{defn}

\begin{prop} \label{prop:Quat-Sas-G2}
If $(\varphi, \mathsf{A})$ is a quaternion-Sasakian structure on $M$, then \\
\indent (a) $\varphi$ is isometric to a nearly-parallel $\G_2$-structure. \\
\indent (b) $\varphi$ is co-closed. \\
\indent (c) $\mathsf{A}$ is integrable.
\end{prop}

\indent We will prove (a) in $\S$\ref{sec:Nearly-Parallel}.  The proofs of (b) and (c), along with a second proof of (a), are in $\S$\ref{sec:MovingFramesQS}.  More generally, we will show that if $(\varphi, \mathsf{A})$ is quaternion-Sasakian, then $\varphi$ belongs to a canonical $2$-parameter family of co-closed $\G_2$-structures $\varphi_{a,b}$ with $\varphi_{1,1} = \varphi$.  Note that this generalization of (b) is implicit in \cite[$\S$5]{friedrich1997nearly}, and that part (c) is a standard fact from $3$-Sasakian geometry.  However, the language of quaternion-Sasakian structures is new.

\subsubsection{Associative $3$-folds in $\SO(4)$-Structures}

\indent \indent The Grassmannian of associative $3$-planes in $\R^7$ is the symmetric space $\text{Asoc} = \G_2/\SO(4)$. In order to study associative 3-folds in 7-manifolds with $\SO(4)$-structure, it is important to understand the $\SO(4)$-action on $\text{Assoc}$. This action is the isotropy action of a rank two symmetric space, so it has cohomogeneity two. The principal orbits are equivariantly diffeomorphic to the $6$-manifold $\SO(4)/(\Z_2 \times \Z_2)$. The $2$-dimensional orbit space can be viewed as a right triangle in $\R^2$, as demonstrated in the following proposition.

\begin{prop}\label{prop:assocorbits}
Let $\SO(4)$ act on $\R^7$ as the stabilizer of the pair $(\varphi, A),$ where
\begin{align*}
    \varphi & = e^{123} + e^{145} + e^{167} + e^{246} - e^{257} - e^{347} - e^{356} \\
    A & = \operatorname{span}(e_1, e_2, e_3).
\end{align*}
Any associative 3-plane $P < \R^7$ is $\SO(4)$-equivalent to the associative 3-plane
\begin{equation*}
    P_{s,r} = \operatorname{span}\!\left( \cos\!\left(2 s\right) e_1 + \sin\!\left(2 s\right) e_6, \ \cos\! \left(s - r \right) e_2 + \sin\! \left(s - r \right) e_5, \ \cos\!\left(s+r\right)e_3 + \sin\!\left(s+r\right) e_4 \right)\!,
\end{equation*}
with $s$ and $r$ unique subject to the constraints $s \geq 0,$ $3 s \leq r,$ $r \leq \pi / 2.$
\end{prop}

\begin{proof}
    The intersection $P \cap A^\perp$ has dimension at least one. The $\SO(4)$-action is transitive on the unit sphere in $A^\perp,$ so by acting by an element of $\SO(4)$ we may assume that $e_7 \in P \cap A^\perp$. The subgroup of $\SO(4)$ stabilizing $e_7$ is $\mathrm{O}(3),$ which splits $\R^7$ as $A \oplus B \oplus e_7,$ where $B = \operatorname{span}(e_4, e_5, e_6).$
    
    We now consider $P$ as a subspace of $A \oplus B$ under the action of $\mathrm{O}(3).$ It follows from the general theory of Jordan angles \cite{jordan1875essai} (also known as principal angles) that $P$ is $\mathrm{O}(3)$-equivalent to the 3-plane
    \begin{equation*}
        P_{\gamma_1, \gamma_2, \gamma_3} = \operatorname{span} \left( \cos (\gamma_1 ) e_1 + \sin(\gamma_1) e_6, \cos (\gamma_2 ) e_2 - \sin(\gamma_2) e_5, \cos (\gamma_3 ) e_3 - \sin(\gamma_3) e_4 \right)
    \end{equation*}
    with $\gamma_1, \gamma_2, \gamma_3$ unique subject to $0 \leq \gamma_1 \leq \gamma_2 \leq \gamma_3 \leq \pi / 2.$ The condition that $P_{\gamma_1, \gamma_2, \gamma_3}$ be associative is $\gamma_1 + \gamma_2 + \gamma_3 = \pi.$ We may parameterize such triples by $\gamma_1 = 2 s,$ $\gamma_2 = s + r,$ $\gamma_3 = \pi - s + r.$ The conditions $0 \leq \gamma_1 \leq \gamma_2 \leq \gamma_3 \leq \pi / 2$ become $s \geq 0,$ $3 s \leq r,$ $r \leq \pi / 2.$
\end{proof}

\begin{figure}
    \centering
\begin{tikzpicture}[scale=0.6]
    \tkzDefPoint(12,4){A}
        \tkzDefPoint(12,0){B}
            \tkzDefPoint(0,0){C}
    \tkzLabelPoint[left](C){$(0,0)$}
    \tkzLabelPoint[right](B){$(\frac{\pi}{2},0)$}
    \tkzLabelPoint[right](A){$(\frac{\pi}{2},\frac{\pi}{6})$}
    \tkzDrawSegment(A,B)
    \tkzDrawSegment(A,C)
    \tkzDrawSegment[style=red](C,B)
    \tkzLabelLine[pos=0.5, below ](C,B){$s = 0$}
    \tkzDrawPoints(A,B,C)
\end{tikzpicture}
\caption{Orbit space of the $\SO(4)$-action on $\text{Asoc} = \G_2/\SO(4)$ in the $rs$-plane.}
\label{fig-triangle}
\end{figure}
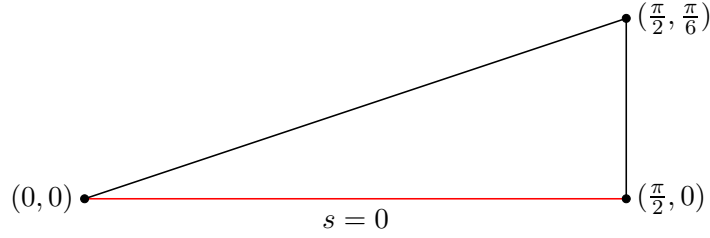

If $N$ is an associative 3-fold in a 7-manifold with $\SO(4)$-structure $(M, \varphi, \mathsf{A}),$ then Proposition \ref{prop:assocorbits} defines two functions $r \colon N \to \left[ 0 ,\pi / 2 \right]$ and $s \colon N \to \left[ 0 ,\pi / 6 \right]$ that describe the position of the tangent planes $T_xN$ relative to the associative distribution $\mathsf{A}_x$.

\begin{prop} \label{prop:AssocIntersect} Let $(M, \varphi, \mathsf{A})$ be a $7$-manifold with an $\SO(4)$-structure, let $N \subset M$ be an associative $3$-fold, and let $x \in N$.  The following are equivalent: \\
\indent (i) $\dim(T_xN \cap \mathsf{A}_x) \geq 1$. \\
\indent (ii) $\dim(T_xN \cap \mathsf{A}_x) = 1$ or $T_xN = \mathsf{A}_x$. \\
\indent (iii) $s(x) = 0$. \\
Moreover, $\dim(T_xN \cap \mathsf{A}_x) = 1$ if and only if $s(x) = 0$ and $r(x) \neq 0$.
\end{prop}

\begin{defn} \label{def:striped} Let $(M, \varphi, \mathsf{A})$ be a $7$-manifold with an $\SO(4)$-structure. An associative $3$-fold $N \subset M$ is called \emph{striped} if either of the following equivalent conditions holds: \\
\indent (i) $\dim(T_xN \cap \mathsf{A}_x) = 1$ for each $x \in \Sigma$. \\
\indent (ii) The function $s\colon N \to \left[ 0 ,\pi / 6 \right]$ vanishes identically and $r \colon N \to \left[ 0 ,\pi / 2 \right]$ is nowhere vanishing.
\end{defn}

\indent When $M$ is a quaternion-Sasakian $7$-manifold or one of its ``squashings" (as defined in $\S$\ref{sec:SquashNew}), we will show in Proposition \ref{prop:stripedHopf} that an associative submanifold of $M$ is striped if and only if it is ``ruled by Hopf circles" and not tangent to $\mathsf{A}$.

\subsection{$\Sp(1)$-Structures and Associative $3$-folds} \label{sec:Sp(1)-str}

\indent \indent In this work, the $\SO(4)$-structures that we encounter will underlie $\Sp(1)$-structures.  Although we are primarily interested in the $\SO(4)$-structure, it will be occasionally be useful to perform computations in terms of the richer $\Sp(1)$-structure.

\begin{defn} Let $M^7$ be a smooth $7$-manifold.  An \emph{almost contact metric structure} (or \emph{$\U(3)$-structure}) is a triple $(\langle \cdot, \cdot \rangle, \alpha, \mathsf{J})$ consisting of a Riemannian metric $\langle \cdot, \cdot \rangle$, a $1$-form $\alpha \in \Omega^1(M)$, and an endomorphism $\mathsf{J} \in \Gamma(\End(TM))$ such that
\begin{align*}
\left.\mathsf{J}^2\right|_{\text{Ker}(\alpha)} & = -\Id & \mathsf{J}(A) & = 0
& \langle \mathsf{J}X, \mathsf{J}Y \rangle & = \langle X,Y \rangle - \alpha(X)\alpha(Y),
\end{align*}
where $A := \alpha^\sharp \in \Gamma(TM)$ is the Reeb field, so that $\alpha(A) = 1$. 

\indent An \emph{$\Sp(1)$-structure} on $M$ consists of data $(\langle \cdot, \cdot \rangle, (\alpha_1, \alpha_2, \alpha_3), (\mathsf{J}_1, \mathsf{J}_2, \mathsf{J}_3))$ for which each triple ($\langle \cdot, \cdot \rangle, \alpha_p, \mathsf{J}_p)$ is an almost contact metric structure, and
\begin{align*}
\mathsf{J}_p \circ \mathsf{J}_q - \alpha_p \otimes A_q & = -\epsilon_{pqr}\mathsf{J}_r - \delta_{pq}\Id \\
\mathsf{J}_p(A_q) & = -\epsilon_{pqr}A_r,
\end{align*}
\end{defn}
\noindent where $A_p := \alpha_p^\sharp \in \Gamma(TM)$ are the Reeb fields.  It is known that a compact smooth $7$-manifold admits an $\Sp(1)$-structure if and only if it is orientable and spin \cite[$\S$3]{friedrich1997nearly}. \\

\indent Thus, given an $\Sp(1)$-structure on $M$, each tangent space can be split as $T_xM = \R A_p \oplus \Ker(\alpha_p)$, for any $p = 1,2,3$, and each $6$-plane $\Ker(\alpha_p)$ carries a special Hermitian structure $(\langle \cdot, \cdot \rangle, \mathsf{J}_p, \Omega_p, \Upsilon_p)$, where $\Omega_p := \langle \mathsf{J}_p \cdot, \cdot \rangle$ and $\Upsilon_p$ is a complex volume form on $\Ker(\alpha_p)$ that is type $(3,0)$ with respect to $\mathsf{J}_p$.  In fact, the tangent spaces decompose further as $T_xM = \R A_1 \oplus \R A_2 \oplus \R A_3 \oplus \mathsf{C}$, where the $4$-plane $\mathsf{C} := \Ker(\alpha_1, \alpha_2, \alpha_3)$ carries a special Hermitian structure in the natural way. \\  

\indent We make four remarks on $\Sp(1)$-structures.  First, since $\Sp(1) \leq \SO(4)$, every $\Sp(1)$-structure has an underlying $\SO(4)$-structure $(\varphi, \mathsf{A})$.  Explicitly, 
\begin{align*}
\varphi & := \alpha_1 \wedge \alpha_2 \wedge \alpha_3 - \alpha_1 \wedge \Omega_1 - \alpha_2 \wedge \Omega_2 - \alpha_3 \wedge \Omega_3 &
\mathsf{A} & := \text{span}(A_1, A_2, A_3).
\end{align*}
Second, note that an $\Sp(1)$-structure determines an entire $2$-sphere of almost contact metric structures.  Indeed, for each $w \in S^2 = \{(w_1, w_2, w_3) \in \R^3 \colon |w| = 1\}$, defining $\alpha_w := w_p\alpha_p$ and $\mathsf{J}_w := w_p\mathsf{J}_p$ yields an almost contact metric structure $(\langle \cdot, \cdot \rangle, \alpha_w, \mathsf{J}_w)$.  \\
\indent Third, if a $7$-manifold $M$ carries an $\Sp(1)$-structure $(g, (\alpha_1, \alpha_2, \alpha_3), (\mathsf{J}_1, \mathsf{J}_2, \mathsf{J}_3))$, then the Riemannian $8$-manifold $(\mathrm{C}(M), g_{\mathrm{C}})$ admits a triple $(I_1, I_2, I_3)$ of $g_{\mathrm{C}}$-compatible almost-complex structures with $I_1I_2 = I_3$ given by
\begin{align*}
I_pX & = \textstyle \mathsf{J}_pX - \alpha_p(X) r \partial_r, \ \ \text{ for } X \in TM, \\
\textstyle I_p(r\partial_r) & = A_p.
\end{align*}
In other words, $\mathrm{C}(M)$ inherits a conical $\Sp(2)$-structure (or \emph{almost hyper-Hermitian structure}).  Conversely, if an $8$-dimensional cone $\mathrm{C}(M)$ admits an $\Sp(2)$-structure $(g_{\mathrm{C}},(\omega_1, \omega_2, \omega_3), (I_1, I_2, I_3) )$ for which each $I_p(r\partial_r)$ is tangent to $M$ and each $\mathscr{L}_{r\partial_r}I_p = 0$, then its link $M$ inherits an $\Sp(1)$-structure via
\begin{align*}
\alpha_p & = \textstyle \partial_r\,\lrcorner\,\omega_p &
\mathsf{J}_p & = \begin{cases} I_p & \mbox{ on }\text{Ker}(\alpha_p) \\
0 & \mbox{ on } \R A_p.
\end{cases}
\end{align*}
This leads to the following definition:

\begin{defn} An $\Sp(1)$-structure on $M$ is called \emph{$3$-Sasakian} if the induced $\Sp(2)$-structure on $\mathrm{C}(M)$ is hyperk\"{a}hler.
\end{defn}

\begin{rmk} \label{rmk:QuatSas23Sas}
    If an $\Sp(1)$-structure on $M$ is 3-Sasakian, then its underlying $\SO(4)$-structure is quaternion-Sasakian. Conversely, a quaternion-Sasakian $\SO(4)$-structure together with a trivialization $\left(w_1, w_2, w_3 \right)$ of the bundle $\mathsf{A}$ gives rise to a unique 3-Sasakian $\Sp(1)$-structure sharing the same metric $g$ and with $\alpha_i = w_i^\flat.$
\end{rmk}

\begin{rmk} On an $8$-manifold $Q$, every $\Sp(2)$-structure $(g, (I_1, I_2, I_3))$ induces an $\Sp(2)\Sp(1)$-structure, essentially by forgetting the triple $(I_1, I_2, I_3)$ and retaining only the rank $3$ subbundle $\text{span}(I_1, I_2, I_3) \subset \End(TQ)$.  In this way, for example, every hyperk\"{a}hler structure underlies a quaternion-K\"{a}hler structure. This passage from an $\Sp(2)$- to an $\Sp(2)\Sp(1)$-structure on an $8$-manifold is analogous to the passage from an $\Sp(1)$- to an $\SO(4)$-structure on a $7$-manifold. 
\end{rmk}

\indent Fourth, we note that $7$-manifolds with $\Sp(1)$-structures admit several distinguished classes of submanifolds.  Here are some of them:

\begin{defn} Let $N^k \subset M^7$ be a $k$-dimensional submanifold for $k = 1,3,5$. \\
\indent Say $N$ is \emph{$\mathsf{J}_1$-CR} if each tangent space $T_xN$ contains the Reeb vector $A_1$ and is $\mathsf{J}_1$-invariant. An equivalent condition is that $\mathrm{C}(N)$ is an $I_1$-complex submanifold of $\mathrm{C}(M)$.
\end{defn}

\begin{defn} Let $N^3 \subset M^7$ be a $3$-dimensional submanifold.
\begin{itemize}
    \item Say $N$ is \emph{$\mathsf{J}_1$-Legendrian} if $\alpha_1|_N = 0$ and $\Omega_1|_N = 0$.  An equivalent condition is that $\mathrm{C}(N)$ is an $I_1$-Lagrangian $4$-fold in $\mathrm{C}(M)$.
    \item Say $N$ is \emph{$\mathsf{J}_1$-special Legendrian} if $\alpha_1|_N = 0$, $\Omega_1|_N = 0$, and $\text{Im}(\Upsilon_1)|_N = 0$. An equivalent condition is that $\mathrm{C}(N)$ is an $I_1$-special Lagrangian $4$-fold in $\mathrm{C}(M)$. 
    \item Say $N$ is \emph{$\mathsf{J}_1$-complex Legendrian} if it is $\mathsf{J}_1$-CR and $\mathsf{J}_2$-Legendrian and $\mathsf{J}_3$-Legendrian.  An equivalent condition is that $\mathrm{C}(N)$ is an $I_1$-complex Lagrangian $4$-fold in $\mathrm{C}(M)$.
\end{itemize}
\end{defn}

\indent The relevance of these submanifolds to our work comes from the following basic observation:

\begin{prop} \label{prop:trivial-associatives} Let $N^3 \subset M^7$ be a $3$-dimensional submanifold, where $M$ carries an $\Sp(1)$-structure.  If $N$ is $\mathsf{J}_w$-CR or $\mathsf{J}_w$-special Legendrian for some $w \in S^2$, then $N$ is an associative $3$-fold (with respect to the underlying $\G_2$-structure on $M$).
\end{prop}

When $M$ carries an $\Sp(1)$-structure, we are interested in constructing associative $3$-folds in $M$ that are \emph{neither} CR nor special Legendrian with respect to any $\mathsf{J}_w$ for $w \in S^2$. The compact examples we construct are obtained by ``twisting" complex Legendrian $3$-folds by a meromorphic function.

\subsection{Nearly-Parallel $\G_2$-Structures} \label{sec:Nearly-Parallel}

\indent \indent We now return to $\G_2$ geometry.  Within the class of co-closed $\G_2$-structures, the subclass of nearly-parallel $\G_2$-structures is particularly important in view of its connection to $\Spin(7)$ geometry.  Indeed, as the following proposition shows, nearly-parallel $\G_2$-structures arise naturally on links of $8$-dimensional cones equipped with conical torsion-free $\Spin(7)$-structures.

\begin{prop}[\cite{bar1993real}] Let $M^7$ be a smooth $7$-manifold.  There is a bijection:
\begin{align*}
\{\text{nearly-parallel }\G_2\text{-str. on }M \text{ with }\lambda = 4\} & \longleftrightarrow \{\text{conical torsion-free } \Spin(7)\text{-str. on }\mathrm{C}(M)\} \\
\varphi & \longmapsto r^3\,dr \wedge \varphi + r^4\ast\! \varphi \\
\textstyle \left.(\partial_r \,\lrcorner\,\Phi)\right|_M & \longmapsfrom \Phi
\end{align*}
\end{prop}

\indent Now, in light of the inclusions $\Sp(2) \leq \SU(4) \leq \Spin(7)$, there is a hierarchy of nearly-parallel metrics on $M$ according to the holonomy group of the $8$-manifold $(\text{C}(M), g_{\mathrm{C}})$.  Precisely:

\begin{defn} Let $(M^7, g)$ be a Riemannian $7$-manifold.
\begin{itemize}
    \item Say $g$ is \emph{nearly-parallel} if $\Hol(g_{\mathrm{C}}) \leq \Spin(7)$.  Such metrics are of the form $g = g_\varphi$ for some nearly-parallel $\G_2$-structure $\varphi$ on $M$.  We say $g$ is \emph{proper nearly-parallel} if $\Hol(g_{\mathrm{C}}) = \Spin(7)$.
    \item Say $g$ is \emph{Sasaki-Einstein} if $\text{Hol}(g_{\mathrm{C}}) \leq \SU(4)$.  Such metrics underlie a Sasaki-Einstein $\SU(3)$-structure on $M$.
    \item Say $g$ is \emph{3-Sasakian} if $\text{Hol}(g_{\mathrm{C}}) \leq \Sp(2)$.  Such metrics underlie a $3$-Sasakian structure on $M$.
\end{itemize}
\end{defn}

\begin{rmk} Note that the structure groups $\Sp(1) \leq \SU(3) \leq \G_2$ are the stabilizers of a unit vector in $\R^8$ under the standard actions of $\Sp(2) \leq \SU(4) \leq \Spin(7)$, respectively.
\end{rmk}

\indent With that above terminology in place, we now prove:

\begin{proof}[Proof of Proposition \ref{prop:Quat-Sas-G2}(a)] Let $(\varphi, \mathsf{A})$ be a quaternion-Sasakian structure on $M$.  By definition, the $8$-manifold $\mathsf{C}(M)$ admits a conical quaternion-K\"{a}hler structure, so the cone metric $g_{\mathrm{C}}$ is quaternion-K\"{a}hler.  Since any conical Einstein metric is Ricci-flat, it follows that $g_{\mathrm{C}}$ is hyperk\"{a}hler, and hence $g$ is $3$-Sasakian, whence $g$ is nearly-parallel.  We conclude there exists a nearly-parallel $\G_2$-structure $\widehat{\varphi}$ whose induced metric is $g$.  That is, $\varphi$ is isometric to the nearly-parallel $\widehat{\varphi}$.
\end{proof}

\indent In the remainder of this section, we briefly discuss the two extreme ends of the hierarchy, namely $3$-Sasakian $7$-manifolds and proper nearly-parallel $7$-manifolds.  Our discussion of $3$-Sasakian $7$-manifolds follows \cite[$\S$13.2 -- $\S$13.5]{boyer2007sasakian}.

\subsubsection{$3$-Sasakian $7$-Manifolds} \label{sec:3-Sasakian}

\indent \indent Let $M^7$ be a compact $7$-manifold equipped with a $3$-Sasakian structure $(\langle \cdot, \cdot\rangle, (\alpha_1, \alpha_2, \alpha_3)$, $(\mathsf{J}_1, \mathsf{J}_2, \mathsf{J}_3))$.  As we now explain, the geometry of $M$ can be understood in terms of a $3$-dimensional foliation $\mathcal{F}_A \subset M$ and an $S^2$-family of $1$-dimensional foliations $\mathcal{F}_w \subset \mathcal{F}_A \subset M$.\\
\indent Since $M$ is compact, the Reeb fields $A_1, A_2, A_3$ are complete.  Moreover, they are an orthonormal set of Killing vector fields, and define a locally free $\Sp(1)$-action on $M$.  In turn, this yields a quasi-regular $3$-dimensional \emph{canonical foliation} $\mathcal{F}_A$ on $M$ whose leaf space $X := M/\mathcal{F}_A$ is a compact $4$-orbifold.  We refer to the leaves of $\mathcal{F}_A$ as \emph{canonical leaves}. 

\begin{thm}[{\cite[$\S$13.3.11, 13.3.13]{boyer2007sasakian}}] Let $M$ be a $3$-Sasakian $7$-manifold such that the Reeb fields $A_1, A_2, A_3$ are complete. \\
\indent (a) The canonical leaves are totally-geodesic, spherical space forms $S^3/\Gamma$ of constant curvature, where $\Gamma \leq \Sp(1)$ is a finite subgroup.  Moreover, the generic leaves are either $S^3$ or $\SO(3)$. \\
\indent (b) The leaf space $X = M/\mathcal{F}_A$ admits the structure of a positive quaternion-K\"{a}hler $4$-orbifold such that the natural projection $h \colon M \to X$ is a principal $G$-orbibundle for $G = \SU(2)$ or $\SO(3)$, and a Riemannian orbifold submersion.  Further, $\alpha = (\alpha_1, \alpha_2, \alpha_3)$ is a connection form for $h$.
\end{thm}

\indent Interestingly, this process can be inverted:

\begin{thm}[{\cite[$\S$13.3.14, 13.3.16]{boyer2007sasakian}} Konishi Inversion] Let $\hat{X}$ be a positive quaternion-K\"{a}hler $4$-orbifold, and let $\epsilon \in H^2_{\mathrm{orb}}(\hat{X}, \Z_2)$ be its Marchiafava-Romani class.  Then: \\
\indent (a) There is a principal $\SO(3)$-orbibundle $\hat{M} \to \hat{X}$ for which $\hat{M}$ admits a $3$-Sasakian structure. \\
\indent (b) The $\SO(3)$-orbibundle $\hat{M} \to \hat{X}$ lifts to a principal $S^3$-orbibundle if and only if $\epsilon = 0$.  In this case, the $3$-Sasakian structure on $\hat{M}$ lifts to the total space of the $S^3$-orbibundle.
\end{thm}

\indent As above, let $M$ be a compact $7$-manifold with a $3$-Sasakian structure.  For each $w \in S^2$, the Reeb field $A_w$ is a complete, unit-length, Killing vector field, and defines a locally free $S^1$-action on $M$.  Thus, $A_w$ yields a quasi-regular $1$-dimensional foliation $\mathcal{F}_w \subset \mathcal{F}_A \subset M$ whose leaves are closed geodesics that we call \emph{Hopf circles}, and whose leaf space $Z_w := M/\mathcal{F}_w$ is a compact $6$-orbifold.  In fact:

\begin{thm}[{\cite[$\S$7.5.1, 13.3.1]{boyer2007sasakian}}]\label{thm:twistspacequot} Let $M$ be a compact $3$-Sasakian $7$-manifold. \\
\indent (a) The projection $p_w \colon M \to Z_w$ is a principal $S^1$-orbibundle with connection $1$-form $\alpha_w$. \\
\indent (b) For $w, w' \in S^2$, there is a diffeomorphism $Z_{w} \approx Z_{w'}$.  In fact, each $Z_w$ may be identified with the (orbifold) twistor space $Z$ of the quaternion-K\"{a}hler 4-orbifold $X = M/\mathcal{F}_A$.
\end{thm}

\begin{thm}[{\cite[$\S$13.2.5, 13.5.10]{boyer2007sasakian}}] Let $M$ be a compact $3$-Sasakian $7$-manifold.  The foliation $\mathcal{F}_{w_0}$ is regular for some $w_0 \in S^2$ if and only if the foliations $\mathcal{F}_w$ are regular for all $w \in S^2$.  In this case, $M$ is homogeneous.
\end{thm}

\begin{rmk}
    It is known that each $Z_w$ admits a K\"{a}hler-Einstein metric $g_{\text{KE}}$ of positive scalar curvature such that the projection $p_w \colon (M,g) \to (Z_w, g_{\text{KE}})$ is an orbifold Riemannian submersion.  However, we will instead equip $Z_w \approx Z$ with its standard \emph{strict nearly-K\"{a}hler structure}.  We will be more explicit in $\S$\ref{sec:NKStr}.
\end{rmk}
Altogether, we have a diagram: 
$$\begin{tikzcd}
M \arrow[dd, "h"'] \arrow[rd, "p_w"] &              \\
                                     & Z \arrow[ld, "\tau"] \\
X                                    &             
\end{tikzcd}$$
\indent Letting $\tau \colon Z \to X$ denote the twistor $S^2$-bundle over $X$, we let $\mathsf{V} \subset TZ$ denote the vertical subbundle, and let $\mathsf{H} \subset TZ$ denote the horizontal subbundle with respect to the nearly-K\"{a}hler metric on $Z$.  Thus, we have an orthogonal splitting $TZ = \mathsf{V} \oplus \mathsf{H}$.  A submanifold of $Z$ that is tangent to $\mathsf{H}$ is said to be \emph{horizontal}.  In terms of the decomposition $T_xM = \R A_1 \oplus \R A_2 \oplus \R A_3 \oplus \mathsf{C}$ discussed in $\S$\ref{sec:Sp(1)-str}, we have $p_w(\mathsf{C}) = \mathsf{H}$ for any $w \in S^2$. \\

\indent We now turn to examples.  The homogeneous $3$-Sasakian $7$-manifolds are classified \cite[$\S$13.4.6]{boyer2007sasakian}, and consist of
\begin{align*}
S^7 & = \frac{\Sp(2)}{\Sp(1)} & \RP^7 & = \frac{\Sp(2)}{\Sp(1) \times \Z_2} & N_{1,1} & = \frac{\SU(3)}{\text{S}(\U(1) \times \U(1))}.
\end{align*}
A wide variety of compact \emph{inhomogeneous} $3$-Sasakian $7$-manifolds have been constructed by Boyer-Galicki-Mann:

\begin{thm}[\cite{mann1994geometry}] For each $(p_1, p_2, p_3) \in \Z^3_+$ with $1 \leq p_1 \leq p_2 \leq p_3$, there exists a compact, simply-connected $3$-Sasakian $7$-manifold called $S(p_1, p_2, p_3)$.  Among them are infinitely many inhomogeneous $7$-manifolds that are not homotopy equivalent.
\end{thm}

\indent Finally, we return to submanifolds.

\begin{prop} \label{prop:trivial-submanifolds} Let $M^7$ be a $3$-Sasakian $7$-manifold, and let $N^3 \subset M^7$ be a $3$-dimensional submanifold.  Fix $w \in S^2$ and consider the $S^1$-orbibundle $p_w \colon M \to Z$, where $Z$ carries its standard strict nearly-K\"{a}hler $\SU(3)$-structure.  Then: \\
\indent (a) $N$ is $\mathsf{J}_w$-CR if and only if $N = p_w^{-1}(S)$ for some pseudo-holomorphic curve $S \subset Z$. \\
\indent (b) $N$ is $\mathsf{J}_w$-complex Legendrian if and only if $N = p_w^{-1}(S)$ for some horizontal pseudo-holomorphic curve $S \subset Z$.
\end{prop}

\subsubsection{Proper Nearly-Parallel $7$-Manifolds}

\indent \indent At the other end of the hierarchy are the \textit{proper} nearly-parallel $7$-manifolds, about which far less is known.  The compact, simply-connected, homogeneous ones have been classified by Friedrich, Kath, Moroianu, and Semmelmann \cite{friedrich1997nearly}, and consist of
\begin{align*}
S^7 & = \frac{\Sp(2)\Sp(1)}{\SO(4)} & B & = \frac{\SO(5)}{\SO(3)}  & N_{k,\ell} & = \frac{\SU(3)}{S^1_{(k,\ell)}}
\end{align*}
In the same paper, the authors constructed \emph{inhomogeneous} proper nearly-parallel $7$-manifolds by ``squashing" the metric of a $3$-Sasakian $7$-manifold.  Precisely:

\begin{thm}[\cite{friedrich1997nearly}, \cite{galicki1996betti}] \label{thm:squashedmetric}  Let $M$ be a $7$-manifold equipped with a $3$-Sasakian structure $(g, (\alpha_1, \alpha_2, \alpha_3)$, $(\mathsf{J}_1, \mathsf{J}_2, \mathsf{J}_3))$.  Orthogonally split $TM = \mathsf{A} \oplus \mathsf{C}$, where $\mathsf{A} = \mathrm{span}(A_1, A_2, A_3)$ and $A_1, A_2, A_3$ are the Reeb fields.  For $t > 0$, define the metrics
$$g_t(X,Y) = \begin{cases}
t^2g(X,Y) & \mbox{if } X,Y \in \mathsf{A} \\
g(X,Y) & \mbox{if } X \in \mathsf{C} \text{ or } Y \in \mathsf{C}.
\end{cases}$$
Then: \\
\indent (a) The metric $g_t$ is Einstein if and only if $t = 1$ or $t = \frac{1}{\sqrt{5}}$. \\
\indent (b) The metric $g_{{1}/{\sqrt{5}}}$ is proper nearly-parallel.
\end{thm}

\begin{rmk}
    The squashed metrics $g_t$ do not require the full $3$-Sasakian structure for their construction, but simply the underlying quaternion-Sasakian structure.
\end{rmk}

\begin{example} ${}$
\begin{enumerate}[label=(\roman*)]
\item Take $M = S^7$ with its standard $\Sp(2)$-invariant 3-Sasakian structure.  The underlying metric $g = g_1$ is the ($\SO(8)$-invariant) round metric.  The proper nearly-parallel metric $g_{{1}/{\sqrt{5}}}$ is the $\Sp(2)\Sp(1)$-invariant Jensen metric \cite{jensen1973einstein} and is induced by an $\Sp(2)\Sp(1)$-invariant nearly-parallel $\G_2$-structure $\varphi$.  The $7$-manifold $(S^7, \varphi, g_{{1}/{\sqrt{5}}})$ is often called the \emph{squashed 7-sphere}.
\item Take $M = N_{1,1}$ with its standard $\SU(3)$-invariant 3-Sasakian structure.  The proper nearly-parallel metric $g_{{1}/{\sqrt{5}}}$ is an $(\SU(3) \times \SO(3))$-invariant Einstein metric induced by an $(\SU(3) \times \SO(3))$-invariant nearly-parallel $\G_2$-structure $\varphi$.  We refer to $7$-manifold $(N_{1,1}, \varphi, g_{{1}/{\sqrt{5}}})$ as a \emph{squashed exceptional Aloff-Wallach space}.
\item Taking $M$ to be an inhomogeneous Boyer-Galicki-Mann $3$-Sasakian $7$-manifold $S(p_1, p_2,$ $p_3)$, Theorem  \ref{thm:squashedmetric}(b) yields compact, inhomogeneous, proper nearly-parallel $7$-manifolds.
\end{enumerate}
\end{example}

\indent Above, we described squashing as a deformation of the \emph{Riemannian structure}.  However, in $\S$\ref{sec:SquashG2}, we will view squashing as a deformation of the underlying \emph{$\G_2$-structure}.  A bit of care will be required in that regard, as distinct $\G_2$-structures can induce the same metric: 
see Example \ref{ex:IsometricS7} and the discussion preceding it.

\begin{rmk} We have discussed various $G$-structures on $7$-manifolds and their $8$-dimensional cones.  As a guide to these, we summarize some of the Lie groups $G$ under consideration:
$$
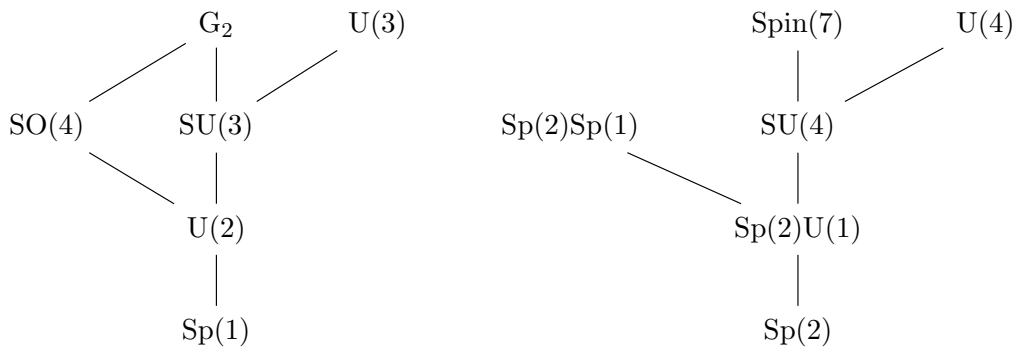
\begin{tikzcd}
                                 & \text{G}_2 \arrow[d, no head] \arrow[ld, no head] & \text{U}(3) \arrow[ld, no head] &                                              & \text{Spin}(7) \arrow[d, no head]          & \text{U}(4) \arrow[ld, no head] \\
\text{SO}(4) \arrow[rd, no head] & \text{SU}(3) \arrow[d, no head]                   &                                 & \text{Sp}(2)\text{Sp}(1) \arrow[rd, no head] & \text{SU}(4) \arrow[d, no head]            &                                 \\
                                 & \text{U}(2) \arrow[d, no head]                    &                                 &                                              & \text{Sp}(2)\text{U}(1) \arrow[d, no head] &                                 \\
                                 & \text{Sp}(1)                                      &                                 &                                              & \text{Sp}(2)                               &                                
\end{tikzcd}$$
Notice that although $\Sp(2)\Sp(1)$ is not a subgroup of $\Spin(7)$, it is nevertheless the case that the $\Sp(2)\Sp(1)$-stabilizer of a unit vector $\mathbf{v} \in S^7 \subset \HH^2$ is contained in the $\Spin(7)$-stabilizer of $\mathbf{v}$, i.e.:
$$\SO(4) = \text{Stab}_{\Sp(2)\Sp(1)}(\mathbf{v}) \leq \text{Stab}_{\Spin(7)}(\mathbf{v}) = \G_2.$$
This ``accidental inclusion" relating $\Sp(2)\Sp(1)$ and $\Spin(7)$ plays a key role in this work: see, e.g., $\S$\ref{sec:SO4-Structures} and the discussion preceding Example \ref{ex:IsometricS7}.
\end{rmk}

\subsection{Hyperk\"{a}hler and Quaternion-K\"{a}hler Manifolds}

\indent \indent In this section, we quickly review the definitions of hyperk\"{a}hler structures, hyperk\"{a}hler metrics, quaternion-K\"{a}hler structures, and quaternion-K\"{a}hler metrics. \\

\indent We begin by recalling the Lie groups $\Sp(n)$ and $\Sp(n)\Sp(1)$. Identifying $\R^{4n}$ with $\mathbb{H}^n,$ multiplication by quaternions $i,$ $j,$ and $k$ gives rise to three complex structures $I_0,$ $J_0$ and $K_0$ on $\R^{4n}$ satisfying the quaternion relation $I_0 J_0 = K_0.$ Lowering an index using the Euclidean metric $g$ gives three symplectic forms $\omega_{I_0},$ $\omega_{J_0},$ and $\omega_{K_0}$ on $\R^{4n}$.   The Lie group $\Sp(n)$ may be defined as the simultaneous $\GL_{4n}(\R)$-stabilizer of these three 2-forms and the group $\Sp(n)\Sp(1)$ may be defined as the $\mathrm{GL}_{4n}(\R)$-stabilizer of the 4-form
$$\Theta_0 = \omega_{I_0}^2 + \omega_{J_0}^2 + \omega_{K_0}^2.$$
One can check that both $\Sp(n)$ and $\Sp(n)\Sp(1)$ fix a positive definite inner product on $\R^{4n}$, so we have $\Sp(n) < \Sp(n)\Sp(1) < \SO(4n)$.  Explicitly, for $n=2$ in an appropriate basis $(e^0, \ldots, e^7)$ of $\R^{8*}$ we have
\begin{equation*}
	\begin{aligned}
		\omega_{I_0} &= e^{01} + e^{23} - e^{45} - e^{67}, \\
		\omega_{J_0} &= e^{02} - e^{13} - e^{46} + e^{57}, \\
		\omega_{K_0} &= e^{03} + e^{12} + e^{47} + e^{56}.
	\end{aligned}
\end{equation*}
While this basis of $\R^{8\ast}$ is non-standard, it has the advantage of aligning with common conventions in $\G_2$ geometry.  \\

\indent We now specialize to real dimension $8$.  Let $Q$ be an $8$-manifold. An \emph{$\Sp(2)$-structure} on $Q$ consists of data $(g, (I_1, I_2, I_3), (\omega_1, \omega_2, \omega_3))$, where $g$ is a Riemannian metric, $(I_1,I_2,I_3)$ are $g$-orthogonal almost-complex structures satisfying the quaternionic relations $I_1I_2 = I_3$, etc., and where $\omega_p := g(I_p \cdot, \cdot)$ are the corresponding K\"{a}hler forms.  A \emph{hyperk\"{a}hler structure} is an $\Sp(2)$-structure for which each $I_p$ is integrable and each $\omega_p$ is closed.  Equivalently, it is a torsion-free $\Sp(2)$-structure.  A \emph{hyperk\"{a}hler metric} on $Q$ is a Riemanninan metric $g$ for which $\Hol(g) \leq \Sp(2)$.  Such metrics are Ricci-flat.  If $(g,(I_p), (\omega_p))$ is a hyperk\"{a}hler structure, then $g$ is a hyperk\"{a}hler metric.  Conversely, if $g$ is a hyperk\"{a}hler metric, then $g$ arises from a hyperk\"{a}hler structure. \\

\indent An \emph{$\Sp(2)\Sp(1)$-structure} on $Q$ is a pair $(g, \mathsf{W})$ consisting of a Riemannian metric $g$ and a rank 3 subbundle $\mathsf{W} \subset \End(TQ)$ such that locally $\mathsf{W}$ admits a basis $\{I, J, K\}$ consisting of $g$-orthogonal almost-complex structures with $IJ = K$.  Equivalently, it is a $4$-form $\Theta \in \Omega^4(Q)$ such that at each $x \in Q$, there is a coframe $L \colon T_xQ \to \R^8$ for which $\Theta|_x = L^*\Theta_0$.  A \emph{quaternion-K\"{a}hler structure} is an $\Sp(2)\Sp(1)$-structure for which the $4$-form $\Theta \in \Omega^4(Q)$ is $g$-parallel.  A \emph{quaternion-K\"{a}hler metric} on $Q$ is a Riemanninan metric $g$ for which $\Hol(g) \leq \Sp(2)\Sp(1)$.  Such metrics are Einstein.  If $(g, \mathsf{W})$ is a QK structure, then $g$ is a QK metric.  Conversely, if $g$ is a QK metric, then $g$ arises from a QK structure. \\

\indent Finally, we explain the relationship between the above objects.  First, note that every $\Sp(2)$-structure $(g, (I_p), (\omega_p))$ induces an $\Sp(2)\Sp(1)$-structure $(g, \mathsf{W})$ by setting $\mathsf{W} := \text{span}(I_1, I_2, I_3)$, or equivalently, setting $\Theta := \omega_1^2 + \omega_2^2 + \omega_3^2$.  Moreover, if the $\Sp(2)$-structure is hyperk\"{a}hler, then the induced $\Sp(2)\Sp(1)$-structure is quaternion-K\"{a}hler and $g$ is Ricci-flat.  Conversely, if a quaternion-K\"{a}hler structure $(g,\mathsf{W})$ is such that $g$ is Ricci-flat, then $\Hol^0(g) \leq \Sp(2)$, so $(g,\mathsf{W})$ is locally induced from a hyperk\"{a}hler structure. \\ 
\indent In particular, suppose that $Q = (\text{C}(M), g_{\mathrm{C}})$ is an $8$-dimensional cone for which $g_{\mathrm{C}}$ is a QK metric.  Since an Einstein cone metric is Ricci-flat, it follows that $g_{\mathrm{C}}$ is Ricci-flat, and hence $g_{\mathrm{C}}$ is a hyperk\"{a}hler metric.  We will exploit this observation in $\S$\ref{sec:MovingFramesQS}.

\section{Squashed $3$-Sasakian $7$-Manifolds} \label{sec:Squashed3Sas}

\indent \indent Let $M$ be a $7$-manifold with an $\SO(4)$-structure $(\varphi, \mathsf{A})$.  In $\S$\ref{sec:SO(4)-Coframe}, we explain how $\varphi$ belongs to a canonical $2$-parameter family of ``squashed" $\G_2$-structures $\varphi_{a,b}$ with $\varphi_{1,1} = \varphi$.  Then, starting in $\S$\ref{sec:MovingFramesQS}, we restrict to the situation where $(\varphi, \mathsf{A})$ is quaternion-Sasakian and study the squashings $\varphi_{a,b}$.  In that case, we show that:
\begin{enumerate}[label=(\roman*)]
    \item  $\mathsf{A}$ integrates to a foliation $\mathcal{F}_A$ whose leaves are associatives with constant positive curvature;
    \item $\varphi_{a,b}$ is co-closed; \item  $\varphi_{a,b}$ is nearly-parallel if and only if $b^2 = 5a^2$; and
    \item $\varphi_{1,1}$ is isometric to a nearly-parallel $\G_2$-structure.
\end{enumerate}

\noindent Note that both (i) and (iv) follow from the discussion in $\S$\ref{sec:Nearly-Parallel}, but we will give a unified exposition of (i)--(iv) in this section.  Moreover, we will see that $\varphi_{t,1}$ induces the squashed metrics $g_t$ of Theorem \ref{thm:squashedmetric}, in which $g_1$ is $3$-Sasakian and $g_{1/\sqrt{5}}$ is proper nearly-parallel.  For this reason, we refer to each $(M, \varphi_{a,b}, \mathsf{A})$ as a ``squashed $3$-Sasakian $7$-manifold" or ``squashed quaternion-Sasakian $7$-manifold." \\

\indent To establish the above claims, in $\S$\ref{sec:MovingFramesQS} we derive the \emph{structure equations} of a squashed $3$-Sasakian $7$-manifold. These equations encode not only the geometry of $M$, but also that of auxiliary spaces associated to $M$.  For this work, the most important such auxiliary space is the $8$-orbifold $Z \times S^2$, where $Z$ is the twistor space of the quaternion-K\"{a}hler $4$-orbifold $X := M/\mathcal{F}_A$.  In $\S$\ref{sec:HopfCircle}, we show that $Z \times S^2$ parametrizes the Hopf circles in $M$, a class of closed geodesics that we will use to construct associative $3$-folds in $M$.  We also show that $Z \times S^2$ has a natural $2$-parameter family of almost-Hermitian structures $(h_{a,b}, J, \omega_{a,b})$; these will play a role in Theorem \ref{thm:MainCorrespondence}.

\subsection{$\SO(4)$-Structures and their Squashings} \label{sec:SO(4)-Coframe}

\indent \indent Let $M$ be a $7$-manifold with an $\SO(4)$-structure $(\varphi, \mathsf{A}$).  As we now explain, $(\varphi, \mathsf{A})$ determines an $\SO(4)$-subbundle of the coframe bundle of $M$.  Indeed, by the discussion in $\S$\ref{sec:SO4-Structures}, the cone $Q := \text{C}(M) = \R^+ \times M$ admits a conical $\Sp(2)\Sp(1)$-structure $\Theta \in \Omega^4(Q)$ via
\begin{equation} \label{eq:ConicalQK}
    \Theta = r^3dr \wedge \left( 4 \, \vol_{\mathsf{A}} - \varphi \right) + r^4\! \ast\!(4 \, \vol_{\mathsf{A}} - \varphi).
    \end{equation}
Since the Lie group $\Sp(2)\Sp(1)$ acts transitively on the unit sphere $S^7 \subset \HH^2$, in an open neighborhood of any point $x \in M$, there exists a coframe $(e^1, \ldots, e^7)$ of $M$ for which
\begin{align} \label{eq:SO(4)-coframe}
    \mathsf{A} & = \text{span}(e_1, e_2, e_3) \notag \\
    \mathsf{C} & = \text{span}(e_4, e_5, e_6, e_7) \\
    \left.\Theta\right|_M & = 3e^{4567} - e^{12} \wedge (e^{45} + e^{67}) - e^{23} \wedge (e^{46} - e^{57}) - e^{31} \wedge (-e^{47} - e^{56}). \notag
\end{align}
Let $\mathcal{B} \subset FM$ denote the bundle of coframes for which (\ref{eq:SO(4)-coframe}) holds. Since the $\Sp(2)\Sp(1)$-stabilizer of a point in $S^7$ is $\SO(4)$, it follows that $\pi \colon \mathcal{B} \to M$ is an $\SO(4)$-bundle.  In terms of a local $\SO(4)$-coframe $(e^1, \ldots, e^7)$, the $\G_2$-structure $\varphi \in \Omega^3(M)$ is given by
\begin{align*}
    \varphi & = e^{123} + e^1 \wedge (e^{45} + e^{67}) + e^2 \wedge (e^{46} - e^{57}) + e^3 \wedge (-e^{47} - e^{56}).
\end{align*}
We may also reconstruct the $4$-form on $Q = \R^+ \times M$ in terms of a local $\SO(4)$-coframe $(e^1, \ldots, e^7)$ via
\begin{equation*}
    \Theta = r^3dr \wedge \left[3 \, e^{123} - e^1 \wedge (e^{45} + e^{67}) - e^2 \wedge (e^{46} - e^{57}) - e^3 \wedge (-e^{47} - e^{56}) \right] + r^4\Theta|_M
\end{equation*}
\indent In computations, we will frequently work on the total space $\mathcal{B}$ or $\mathcal{B} \times \R$ and omit pullbacks from the notation.  For example, let $(\alpha,\beta) = (\alpha_1, \alpha_2, \alpha_3, \beta_1, \beta_2, \beta_3, \beta_4) \in \Omega^1(\mathcal{B}; \R^7)$ denote the tautological $1$-form. Pulled back to $\mathcal{B} \times \R$, the Riemannian metrics on $M$ and $Q$ are, respectively,
\begin{align*}
    g & = \alpha_1^2 + \alpha_2^2 + \alpha_3^2 + \beta_1^2 + \beta_2^2 + \beta_3^2 + \beta_4^2 \\
    g_Q & = dr^2 + r^2(\alpha_1^2 + \alpha_2^2 + \alpha_3^2 + \beta_1^2 + \beta_2^2 + \beta_3^2 + \beta_4^2)
\end{align*}
Similarly, continuing to omit pullbacks from the notation, the conical $\Sp(2)\Sp(1)$-structure $\Theta$ is
$$\Theta = r^3\,dr \wedge \chi_3 + r^4\chi_4$$
where
\begin{align*}
    \chi_3 & := 3 \, \alpha_{123} - \alpha_1 \wedge (\beta_{12} + \beta_{34}) - \alpha_2 \wedge (\beta_{13} - \beta_{24}) - \alpha_3 \wedge (-\beta_{14} - \beta_{23}) \\
    \chi_4 & := 3 \,\beta_{1234} - \alpha_{23} \wedge (\beta_{12} + \beta_{34}) - \alpha_{31} \wedge (\beta_{13} - \beta_{14}) - \alpha_{12} \wedge (-\beta_{14} - \beta_{23}).
\end{align*}

\subsubsection{Squashed $\G_2$-Structures} \label{sec:SquashNew}

\indent \indent On the total space $\mathcal{B}$, the set of $\SO(4)$-invariant semibasic 3-forms is spanned by
\begin{align*}
	\Psi_1 & = \alpha_{123}, & \Psi_2 & = \alpha_1 \wedge \left(\beta_{12} + \beta_{34} \right) + \alpha_2 \wedge \left(\beta_{13} - \beta_{24} \right) + \alpha_3 \wedge \left(-\beta_{14} - \beta_{23} \right)\!,
\end{align*}
while the set of $\SO(4)$-invariant semibasic 4-forms on $\mathcal{B}$ is spanned by
\begin{align*}
	\Gamma_1 & = \beta_{1234}, & \Gamma_2 & = \alpha_{23} \wedge \left(\beta_{12} + \beta_{34} \right) + \alpha_{31} \wedge \left(\beta_{13} - \beta_{24} \right) + \alpha_{12} \wedge \left(-\beta_{14} - \beta_{23} \right)\!.
\end{align*} 
Let $a, b \in \R$ be non-zero constants and define a 3-form $\varphi_{a,b}$ on $\mathcal{B}$ by
\begin{equation} \label{eq:SquashedG2}
	\varphi_{a,b} = a^3 \Psi_1 + a b^2 \Psi_2.
\end{equation}
The $3$-form $\varphi_{a,b}$ is $\SO(4)$-invariant and descends to $M$ to define a $\G_2$-structure (which will also be denoted by $\varphi_{a,b}$). The resulting metric and 4-form will be denoted by $g_{a,b}$ and $\psi_{a,b}$ respectively. On $\mathcal{B}$,
\begin{equation*}
	\begin{aligned}
		g_{a,b} &= a^2 \left(\alpha_1^2 + \alpha_2^2 + \alpha_3^3 \right) + b^2 \left(\beta_1^2 + \beta_2^2 + \beta_3^2 + \beta_4^2 \right)\!, \\
		\psi_{a,b} &= b^4 \Gamma_1 + a^2 b^2 \Gamma_2. 
	\end{aligned}	
\end{equation*}
The $3$-forms $\varphi_{a,b}$ are the \emph{squashed $\G_2$-structures} associated to the $\SO(4)$-structure $(\varphi, \mathsf{A})$.

\subsection{Quaternion-Sasakian $7$-Manifolds and their Squashings} \label{sec:MovingFramesQS}

\indent \indent Let $M$ be a $7$-manifold equipped with a quaternion-Sasakian structure $(\varphi, \mathsf{A})$, and let $Q := \text{C}(M) = \R^+ \times M$ be the metric cone of $M$.  By definition, the conical $\Sp(2)\Sp(1)$-structure $\Theta \in \Omega^4(Q)$ defined in (\ref{eq:ConicalQK}) is quaternion-K\"{a}hler.  In this section, we equip the total space $\mathcal{B}$ of the $\SO(4)$-coframe bundle of $M$ with a global coframe and derive the corresponding structure equations --- i.e., we will express the exterior derivatives of the coframe in terms of the coframe itself. \\

\indent Let $(\alpha,\beta) = (\alpha_1, \alpha_2, \alpha_3, \beta_1, \beta_2, \beta_3, \beta_4) \in \Omega^1(\mathcal{B}; \R^7)$ denote the tautological $1$-form.  This provides a basis for the semibasic (or ``horizontal") $1$-forms on $\mathcal{B}$, and we aim to complete $(\alpha, \beta)$ to a coframe of $\mathcal{B}$.  For this, let $\mathcal{F}_M$ be the $\SO(7)$-bundle of oriented $g$-orthonormal coframes, and note that $\mathcal{B} \subset \mathcal{F}_M$.  The Levi-Civita form of $g$, denoted $\psi_M \in \Omega^1(\mathcal{F}_M; \mathfrak{so}(7))$, is a connection form for $\mathcal{F}_M \to M$.  Its restriction $\left.\psi_M\right|_{\mathcal{B}} \in \Omega^1(\mathcal{B}; \mathfrak{so}(7))$ is not a connection for $\mathcal{B} \to M$, but splits into connection terms and semibasic terms.  Indeed, noting the $\SO(4)$-irreducible decomposition
$$\mathfrak{so}(7) = \mathsf{V}_{0,2} \oplus \mathsf{V}_{1,1} \oplus \mathsf{V}_{1,3} \oplus \mathfrak{sp}(1) \oplus \mathfrak{sp}(1),$$
where $\mathsf{V}_{0,2}$, $\mathsf{V}_{1,1}$, and $\mathsf{V}_{1,3}$ are irreducible $\SO(4)$-modules of dimensions $3$, $4$, $8$, respectively, the restriction $\psi_M|_{\mathcal{B}} \in \Omega^1(\mathcal{B}; \mathfrak{so}(7))$ splits into blocks as
	\begin{equation*}
		\left.\psi_M\right|_{\mathcal{B}} = \left[\begin{array}{c|c}
			\left[ \gamma + 2 \zeta \right]_3 & \sigma + \rho \\ \hline
			- \sigma^t - \rho^t & \left[ \zeta \right]^+ + \left[ \nu \right]^-
		\end{array}\right],
	\end{equation*}
	where
	\begin{equation*}
		\begin{aligned}
			\left[\gamma + 2 \zeta \right]_3 &= \begin{bmatrix}
				0 & \gamma_3 + 2 \zeta_3 & -\gamma_2 - 2 \zeta_2 \\
				-\gamma_3 - 2 \zeta_3 & 0 & \gamma_1 +  2 \zeta_1 \\
				\gamma_2 + 2 \zeta_2 & -\gamma_1 - 2 \zeta_1 & 0
			\end{bmatrix}, \\
			\sigma + \rho &= \left[ \begin {array}{cccc} 2\,\sigma_{{6}} - \rho_{{2}} & -2 \, \sigma_{{5}} + \rho_{{1}} & 2\,\sigma_{{8}}- \rho_{{4}} & -2\,\sigma_{{7}} + \rho_{{3}} \\
			-\sigma_{{1}} - \sigma_{{7}} - \rho_{{3}} & -\sigma_{{2}} + \sigma_{{8}} + \rho_{{4}} & -\sigma_{{3}} + \sigma_{{5}} + \rho_{{1}} & -\sigma_{{4}} -\sigma_{{6}} - \rho_{{2}} \\ \sigma_{{2}} + \sigma_{{8}} + \rho_{{4}} & -\sigma_{{1}} + \sigma_{{7}} + \rho_{{3}} & \sigma_{{4}} - \sigma_{{6}} - \rho_{{2}} & -\sigma_{{3}}-\sigma_{{5}}-\rho_{{1}} \end {array} \right], \\
	\left[ \zeta \right]^+ + \left[ \nu \right]^- &=  \begin{bmatrix}
	 	0 & -\zeta_1 - \nu_1 & - \zeta_2 + \nu_2 & \zeta_3 + \nu_3 \\
	 	\zeta_1 + \nu_1 & 0 & \zeta_3 - \nu_3 & \zeta_2 + \nu_2 \\
	 	\zeta_2 - \nu_2 & - \zeta_3 + \nu_3 & 0 & \zeta_1 + \nu_1 \\
	 	-\zeta_3 - \nu_3 & \zeta_2 - \nu_2 & \zeta_1 - \nu_1 & 0
	 \end{bmatrix}.
 	\end{aligned}
	\end{equation*}
Here, $\zeta$ and $\nu$ are the components of the $\mathfrak{so}(4) = \mathfrak{sp}(1) \oplus \mathfrak{sp}(1)$-valued \emph{natural connection} on $\mathcal{B}$, whereas $\gamma \in \Omega^1(\mathcal{B}; \mathsf{V}_{0,2})$, $\rho \in \Omega^1(\mathcal{B}; \mathsf{V}_{1,1})$, and $\sigma  \in \Omega^1(\mathcal{B}; \mathsf{V}_{1,3})$ are semi-basic $1$-forms.  The components of $(\alpha, \beta, \zeta, \nu) \in \Omega^1(\mathcal{B}; \R^7 \oplus \mathfrak{so}(4))$ provide a coframe for $\mathcal{B}$. \\  
	
	Before computing the exterior derivatives $d\alpha$, $d\beta$, $d\zeta$, $d\nu$, we consider the Levi-Civita connection $\psi_Q \in \Omega^1(\mathcal{F}_Q; \mathfrak{so}(8))$ of $g_Q$, where $\mathcal{F}_Q \to Q$ is the $\SO(8)$-bundle of oriented $g_Q$-orthonormal coframes.  Restricting to the link $M \simeq \{1\} \times M \subset Q$, the $1$-form $\psi_Q$ is given by
	\begin{equation*}
		\psi_Q = \left[\begin{array}{c|c|c}
			0 & -\alpha & -\beta \\ \hline
			\alpha & \left[ \gamma + 2 \zeta \right]_3 & \sigma + \rho \\ \hline
			\beta & - \sigma^t - \rho^t & \left[ \zeta \right]^+ + \left[ \nu \right]^-
		\end{array}\right]\!.
	\end{equation*}
	Since the $4$-form $\Theta \in \Omega^4(Q)$ is parallel with respect to the Levi-Civita connection on $Q$, it follows that $\psi_Q$ must take values in the Lie algebra $\mathfrak{sp}(2) \oplus \mathfrak{sp}(1)$.  Therefore, $\gamma = \alpha,$ $\rho = \beta$ and $\sigma = 0$.  We may now compute:

\begin{subequations}\label{eq:3SasakStruct}
\begin{prop}\label{prop:structeqns} The \emph{first structure equations} are given by:
	\begin{small}
			\begin{align} 
				d \begin{bmatrix}
					\alpha_1 \\
					\alpha_2 \\
					\alpha_3
				\end{bmatrix} &= - \begin{bmatrix}
					0 & 2 \zeta_3 & -2 \zeta_2 \\
					-2 \zeta_3 & 0 & 2 \zeta_1 \\
					2 \zeta_2 & -2 \zeta_1 & 0
				\end{bmatrix} \wedge \begin{bmatrix}
					\alpha_1 \\
					\alpha_2 \\
					\alpha_3
				\end{bmatrix} + 2 \begin{bmatrix}
					\alpha_2 \wedge \alpha_3 - \beta_1 \wedge \beta_2 - \beta_3 \wedge \beta_4 \\
					\alpha_3 \wedge \alpha_1 - \beta_1 \wedge \beta_3 + \beta_2 \wedge \beta_4 \\
					\alpha_1 \wedge \alpha_2 + \beta_1 \wedge \beta_4 + \beta_2 \wedge \beta_3
				\end{bmatrix}, \\
				d \begin{bmatrix}
					\beta_1 \\
					\beta_2 \\
					\beta_3 \\
					\beta_4
				\end{bmatrix} &= - \begin{bmatrix}
					0 & -\zeta_1 - \nu_1 & - \zeta_2 + \nu_2 & \zeta_3 + \nu_3 \\
					\zeta_1 + \nu_1 & 0 & \zeta_3 - \nu_3 & \zeta_2 + \nu_2 \\
					\zeta_2 - \nu_2 & - \zeta_3 + \nu_3 & 0 & \zeta_1 + \nu_1 \\
					-\zeta_3 - \nu_3 & \zeta_2 - \nu_2 & \zeta_1 - \nu_1 & 0
				\end{bmatrix} \wedge \begin{bmatrix}
					\beta_1 \\
					\beta_2 \\
					\beta_3 \\
					\beta_4
				\end{bmatrix} + \begin{bmatrix}
					\alpha_1 \wedge \beta_2 + \alpha_2 \wedge \beta_3 - \alpha_3 \wedge \beta_4 \\
					-\alpha_1 \wedge \beta_2 - \alpha_3 \wedge \beta_3 - \alpha_2 \wedge \beta_4 \\
					\alpha_2 \wedge \beta_2 + \alpha_3 \wedge \beta_3 + \alpha_1 \wedge \beta_4 \\
					\alpha_3 \wedge \beta_2 + \alpha_2 \wedge \beta_3 - \alpha_1 \wedge \beta_4
				\end{bmatrix}.
		\end{align}
		\end{small}

\noindent Further, there exist functions $W_{ij} \colon \mathcal{B} \to \R$ with $W_{11} + W_{22} + W_{33} = 0$ such that the \emph{second structure equations} hold:
        \begin{small}
		\begin{align}
				d \begin{bmatrix}
					\zeta_1 \\
					\zeta_2 \\
					\zeta_3 
				\end{bmatrix} &= 2 \begin{bmatrix}
					\zeta_2 \wedge \zeta_3 \\
					\zeta_3 \wedge \zeta_1 \\
					\zeta_1 \wedge \zeta_2
				\end{bmatrix}, \\
				d \begin{bmatrix}
					\nu_1 \\
					\nu_2 \\
					\nu_3
				\end{bmatrix} &= 2 \begin{bmatrix}
					\nu_2 \wedge \nu_3 \\
					\nu_3 \wedge \nu_1 \\
					\nu_1 \wedge \nu_2
				\end{bmatrix} + \begin{bmatrix}
					2 + W_{11} & W_{12} & W_{13} \\
					W_{12} & 2 + W_{22} & W_{23} \\
					W_{13} & W_{23} & 2 + W_{33}
				\end{bmatrix} \begin{bmatrix}
					- \beta_1 \wedge \beta_2 + \beta_3 \wedge \beta_4 \\
					\beta_1 \wedge \beta_3 + \beta_2 \wedge \beta_4 \\
					\beta_1 \wedge \beta_4 - \beta_2 \wedge \beta_4
				\end{bmatrix}.
			\end{align}
	\end{small}
\end{prop}
\end{subequations}

\begin{proof}
By the Fundamental Lemma of Riemannian Geometry, we have
$$d\begin{bmatrix} \alpha \\ \beta \end{bmatrix} = -\psi_M \wedge \begin{bmatrix} \alpha \\ \beta \end{bmatrix}\!.$$
This implies the first structure equations.  Taking exterior derivatives of the first structure equations, noting that $d(d\alpha) = 0$ and $d(d\beta) = 0$ and applying Cartan's Lemma, yields the second structure equations.
\end{proof}

\begin{rmk}\label{rmk:so4structdisc}
	The structure equations (\ref{eq:3SasakStruct}) characterize quaternion-Sasakian structures in the following sense.  If $P$ is a 13-manifold endowed with a coframe $\left( \alpha, \beta, \zeta, \nu \right)$ and a function $W\colon P \to \mathrm{Sym}^2_0 \left( \R^3 \right)$ satisfying (\ref{eq:3SasakStruct}), then the equations $\alpha = \beta = 0$ define an integrable plane field of codimension 7 on $P.$ Let $N$ be the leaf space of the resulting foliation. It is straightforward to check that the pair $\left( \varphi, \mathsf{A} \right)$ given by
	\begin{equation*}
	    \begin{aligned}
	        \varphi &= \alpha_{123} + \alpha_1 \wedge (\beta_{12} + \beta_{34}) + \alpha_2 \wedge (\beta_{13} - \beta_{24}) + \alpha_3 \wedge (-\beta_{14} - \beta_{23})\\
	        \mathsf{A} &= \left\lbrace \beta = 0 \right\rbrace
	    \end{aligned}
	\end{equation*}
	descend to $N$ to define a quaternion-Sasakian structure. The space $P$ may be identified with an open subset of the $\SO(4)$-structure $\mathcal{B}$ associated to this quaternion-Sasakian structure.
	
\end{rmk}

The 1-form $\zeta \in \Omega^1(\mathcal{B}; \mathfrak{sp}(1))$ defines a flat connection on the subbundle $\mathsf{A} \subset TM$.  The vector bundle $\mathsf{A}$ is trivial: it is spanned by the three linearly independent non-vanishing vector fields $I_1(\partial_r)$, $I_2(\partial_r)$, $I_3(\partial_r)$.

\begin{rmk} \label{rmk:zeta-flat}
 Choosing a fixed $\zeta$-flat trivialization of $\mathsf{A}$ defines an $\Sp(1)$-subbundle $\mathcal{B}' \subset \mathcal{B}$ on which $\zeta = 0$.  It is possible to define and work with 3-Sasakian structures using the subbundle $\mathcal{B}'$ alone. However, for our purposes it is preferable to work directly with the $\SO(4)$-structure $\mathcal{B}$ because the $\G_2$-structures we define on $M$ will naturally be invariant under diffeomorphisms of $M$ preserving $\mathcal{B}.$ The group of diffeomorphisms of $M$ preserving $\mathcal{B}$ is strictly larger than the group preserving $\mathcal{B}'$ and our techniques will allow us to determine when two associative submanifolds are equivalent up to the action of the larger group, something that would not be possible if we worked only with a fixed $\zeta$-flat trivialization of $\mathsf{A}$ and the resulting $\Sp(1)$-structure $\mathcal{B}'$.
\end{rmk}

\subsubsection{The $4$-Dimensional Quotient} \label{sec:4dimquotient}

\indent \indent We now observe that the associative distribution $\mathsf{A} \subset TM$ is integrable.  Indeed, consider the distribution $\beta = 0$ on the $13$-manifold $\mathcal{B}$.  The structure equations (\ref{eq:3SasakStruct}) show both that $\beta = 0$ descends to $M$, where it coincides with $\mathsf{A}$, and furthermore that $\mathsf{A}$ is integrable.  The resulting associative foliation of $M$ is the canonical foliation $\mathcal{F}_A \subset M$ discussed in $\S$\ref{sec:3-Sasakian}.  Moreover, it follows from the $d\alpha$ and $d \zeta$ equations in (\ref{eq:3SasakStruct}) that the induced metric on any canonical leaf $F \subset M$ has constant positive curvature, which also accords with the remarks in $\S$\ref{sec:3-Sasakian}. \\

\indent Let $X = M/\mathcal{F}_A$ denote the 4-dimensional leaf space of the distribution $\mathsf{A}$ on $M$, and let $h \colon M \to X$ denote the quotient map.  With respect to the projection $h \circ \pi \colon \mathcal{B} \to X$, the $1$-forms $\beta_1, \beta_2, \beta_3, \beta_4$ provide a basis for the semi-basic forms.  It follows from the structure equations (\ref{eq:3SasakStruct}) that the symmetric 2-tensor
\begin{equation*}
	\beta_1^2 + \beta_2^2 + \beta_3^2 + \beta_4^2 
\end{equation*}
on $\mathcal{B}$ descends to $X$ to define a Riemannian metric $g_X$ on $X$.  The equations for $d\nu, d\zeta$ and $d\alpha$ in (\ref{eq:3SasakStruct}) imply that $g_X$ is self-dual Einstein with positive scalar curvature. The quaternion-Sasakian $7$-manifold $M$ may be locally identified with the $\SO(3)$-frame bundle of the rank 3 Riemannian vector bundle $\Lambda^2_+(T^*X) \to X$ of self-dual $2$-forms on $X$. \\ 

\indent The construction of the preceding paragraph can be reversed. Let $(X,g_X)$ be a Riemannian 4-orbifold, and set $\mathcal{B} := \mathcal{F}_X \times \SO(3)$, where $\mathcal{F}_X$ is the orthonormal coframe bundle of $X$.  There is an $\SO(4)$-action on $\mathcal{B}$ defined as follows.  Letting $\rho \colon \SO(4) \to \SO(3)$ denote the $\SO(4)$-representation on $\Lambda^2_+ \R^4 \simeq \R^3$, an element $k \in \SO(4)$ acts on $\mathcal{F}_X \times \SO(3)$ by the standard action on the first factor and via left-multiplication by $\rho(k)$ on the second factor. The quotient $M := \mathcal{B} / \SO(4)$ is exactly the $\SO(3)$-frame bundle of the rank 3 bundle $\Lambda^2_+(T^*X) \to X$.  \\
\indent If $g_X$ is a self-dual positive Einstein metric, then we can endow $M$ with a quaternion-Sasakian $\SO(4)$-structure in the following way.  The space $\mathcal{B}$ carries the following vector-valued 1-forms $\alpha, \beta, \zeta$, and $\nu$:
\begin{itemize}
    \item $\beta$ is the pullback of tautological 1-form on $\mathcal{F}_X$
    \item $\zeta$ is the pullback of the left-invariant Maurer-Cartan form on $\SO(3)$, and
    \item $\alpha = \mu_2 - \zeta$, $\nu = \mu_1$, where $\mu_1, \mu_2$ denote the two components of the $(\mathfrak{sp}(1) \oplus \mathfrak{sp}(1))$-valued Levi-Civita form of $g_X$.
\end{itemize}
\indent Note that $\alpha$ and $\beta$ are semibasic for the projection $\mathcal{B} \to M$.  Moreover, these forms satisfy the hypotheses and conclusion of Proposition \ref{prop:structeqns}, and therefore by the discussion in Remark \ref{rmk:so4structdisc} endow $M$ with a quaternion-Sasakian $\SO(4)$-structure.

\subsubsection{Squashed $3$-Sasakian $7$-Manifolds} \label{sec:SquashG2}

\indent \indent We continue to let $M$ be a $7$-manifold with a quaternion-Sasakian structure $(\varphi, \mathsf{A})$.  Let $\varphi_{a,b}$ denote the associated squashed $\G_2$-structures (\ref{eq:SquashedG2}) with corresponding $4$-forms $\psi_{a,b}$ and induced metrics $g_{a,b}$.  Using the structure equations (\ref{eq:3SasakStruct}), we may compute
\begin{equation*}
	\begin{aligned}
		d \, \varphi_{a,b} &= - 2 \frac{a^2 + b^2}{a b^2} \psi_{a,b} - 2 \frac{b^2 \left(5 a^2 - b^2 \right)}{a} \Gamma_1, \\
		d \, \psi_{a,b} & = 0.
	\end{aligned}
\end{equation*}
Therefore, each $\G_2$-structure $\varphi_{a,b}$ is co-closed.  Moreover, $\varphi_{a,b}$ is nearly-parallel if and only if $b^2 = 5 a^2$.

\begin{defn} Let $M$ be a $7$-manifold with a quaternion-Sasakian $\SO(4)$-structure $(\varphi, \mathsf{A})$.  We refer to any member of the $2$-parameter family $(M, \varphi_{a,b}, \mathsf{A})$ as a \emph{squashed quaternion-Sasakian $7$-manifold}, or by abuse of language, as a \emph{squashed $3$-Sasakian $7$-manifold}.
\end{defn}

\begin{rmk}
    Renaming $a = t$ and setting $b = 1$, we obtain a $1$-parameter family of $\G_2$-structures $\varphi_t := \varphi_{t,1}$ that induce the squashed metrics $g_t = g_{t,1}$ of Theorem \ref{thm:squashedmetric}.  Note that $\varphi_t$ is nearly-parallel if and only if $t = \frac{1}{\sqrt{5}}$.
\end{rmk}

\begin{figure}
    \centering
\begin{tikzpicture}

\begin{axis}[xmin=0,xmax=4, ymin=0, ymax=4, axis lines = left, xticklabels={,,}, yticklabels={,,}]

\addplot[color=red]{x}
node[right,pos=0.75]{$\ b = a$};

\addplot[color=blue]{2.23*x}
node[left,pos=0.65]{$b = \sqrt{5}a$};

\node[label={315:{$\varphi_{1,1}$}}, circle, fill, inner sep=1pt] at (axis cs:1,1) {};

\end{axis}

\end{tikzpicture}
\caption{The family of co-closed $\G_2$-structures $\varphi_{a,b}$ for $a,b > 0$.}
\label{fig-triangle}
\end{figure}
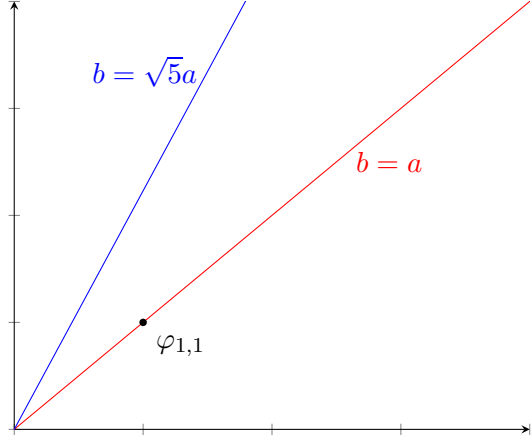

\indent Now, since the $\SO(4)$-structure $(\varphi, \mathsf{A})$ is quaternion-Sasakian, Proposition \ref{prop:Quat-Sas-G2}(a) implies that the co-closed $\G_2$-structure $\varphi = \varphi_{1,1}$ is \emph{isometric} to some nearly-parallel $\G_2$-structure $\widehat{\varphi}$.  We can now explicitly exhibit such a $\G_2$-structure.  Indeed, by choosing a trivialization of $\mathsf{A}$ as in Remark \ref{rmk:zeta-flat}, the $3$-form $\widehat{\varphi} \in \Omega^3(\mathcal{B}')$ given by
\begin{equation*}
		\widehat{\varphi} := \alpha_{123} - \alpha_1 \wedge \left(\beta_{12} + \beta_{34} \right) + \alpha_2 \wedge \left(\beta_{13} - \beta_{24} \right) + \alpha_3 \wedge \left(-\beta_{14} - \beta_{23} \right),
	\end{equation*}
which differs from $\varphi_{1,1}$ by a sign, descends to a nearly-parallel $\G_2$-structure on $M$ that induces $g = g_{1,1}$.  This subtlety is related to the fact that $\Sp(2)\Sp(1)$ is not a subgroup of $\mathrm{Spin}(7)$ (whereas $\Sp(2)$ is a subgroup of both).  The following example may help clarify the situation.

\begin{example} \label{ex:IsometricS7} 
Let $M = S^7$ carry its standard $\Sp(2)$-invariant $3$-Sasakian structure.  There is a unique $\Spin(7)$-invariant $\G_2$-structure $\widehat{\varphi} \in \Omega^3(S^7)$ that induces the round metric of constant curvature $1$.  This $\G_2$-structure is nearly-parallel and is called the \emph{standard $\G_2$-structure} on $S^7$.  Note that $\widehat{\varphi}$ does \emph{not} belong to the family of co-closed, $\Sp(2)\Sp(1)$-invariant $\G_2$-structures $\varphi_{a,b}$, but \emph{is} isometric to $\varphi_{1,1}$.
\end{example}

Let $\mu, \theta$ denote rescaled forms on $\mathcal{B}$ adapted to the $\G_2$-structure $\varphi_{a,b}$: 
\begin{equation*}
	\mu = a \alpha, \:\:\:\: \theta = b \beta.
\end{equation*}
In terms of the rescaled forms, we have
\begin{equation*}
	\begin{aligned}
		\varphi_{a,b} &= \mu_{123} + \mu_1 \wedge \left(\theta_{12} + \theta_{34} \right) + \mu_2 \wedge \left(\theta_{13} - \theta_{24} \right) + \mu_3 \wedge \left(-\theta_{14} - \theta_{23} \right), \\
		\psi_{a,b} &= \theta_{1234} + \mu_{23} \wedge \left(\theta_{12} + \theta_{34} \right) + \mu_{31} \wedge \left(\theta_{13} - \theta_{24} \right) + \mu_{12} \wedge \left(-\theta_{14} - \theta_{23} \right), \\
		g_{a,b} &= \mu_1^2 + \mu_2^2 + \mu_3^3 + \theta_1^2 + \theta_2^2 + \theta_3^2 + \theta_4^2.
	\end{aligned}
\end{equation*}
The structure equations (\ref{eq:3SasakStruct}) become
\begin{subequations}\label{eq:Scaled3SasakStruct}
\begin{small}
	\begin{align}
		d \begin{bmatrix}
			\mu_1 \\
			\mu_2 \\
			\mu_3
		\end{bmatrix} &= - \begin{bmatrix}
			0 & 2 \zeta_3 & -2 \zeta_2 \\
			-2 \zeta_3 & 0 & 2 \zeta_1 \\
			2 \zeta_2 & -2 \zeta_1 & 0
		\end{bmatrix} \wedge \begin{bmatrix}
			\mu_1 \\
			\mu_2 \\
			\mu_3
		\end{bmatrix} + \frac{2}{a} \begin{bmatrix}
			\mu_2 \wedge \mu_3  \\
			\mu_3 \wedge \mu_1  \\
			\mu_1 \wedge \mu_2
		\end{bmatrix} + \frac{2}{b^2} \begin{bmatrix}
		 - \theta_1 \wedge \theta_2 - \theta_3 \wedge \theta_4 \\
		- \theta_1 \wedge \theta_3 + \theta_2 \wedge \theta_4 \\
		\theta_1 \wedge \theta_4 + \theta_2 \wedge \theta_3
	\end{bmatrix} , \\
		d \begin{bmatrix}
			\theta_1 \\
			\theta_2 \\
			\theta_3 \\
			\theta_4
		\end{bmatrix} &= - \begin{bmatrix}
			0 & -\zeta_1 - \nu_1 & - \zeta_2 + \nu_2 & \zeta_3 + \nu_3 \\
			\zeta_1 + \nu_1 & 0 & \zeta_3 - \nu_3 & \zeta_2 + \nu_2 \\
			\zeta_2 - \nu_2 & - \zeta_3 + \nu_3 & 0 & \zeta_1 + \nu_1 \\
			-\zeta_3 - \nu_3 & \zeta_2 - \nu_2 & \zeta_1 - \nu_1 & 0
		\end{bmatrix} \wedge \begin{bmatrix}
		\theta_1 \\
		\theta_2 \\
		\theta_3 \\
		\theta_4
	\end{bmatrix} + \frac{1}{a} \begin{bmatrix}
			\mu_1 \wedge \theta_2 + \mu_2 \wedge \theta_3 - \mu_3 \wedge \theta_4 \\
			-\mu_1 \wedge \theta_2 - \mu_3 \wedge \theta_3 - \mu_2 \wedge \theta_4 \\
			\mu_2 \wedge \theta_2 + \mu_3 \wedge \theta_3 + \mu_1 \wedge \theta_4 \\
			\mu_3 \wedge \theta_2 + \mu_2 \wedge \theta_3 - \mu_1 \wedge \theta_4
		\end{bmatrix},
	\end{align}
	\end{small}
	and
    \begin{small}
	\begin{align}
		d \begin{bmatrix}
			\zeta_1 \\
			\zeta_2 \\
			\zeta_3 
		\end{bmatrix} &= 2 \begin{bmatrix}
			\zeta_2 \wedge \zeta_3 \\
			\zeta_3 \wedge \zeta_1 \\
			\zeta_1 \wedge \zeta_2
		\end{bmatrix}, \\
		d \begin{bmatrix}
			\nu_1 \\
			\nu_2 \\
			\nu_3
		\end{bmatrix} &= 2 \begin{bmatrix}
			\nu_2 \wedge \nu_3 \\
			\nu_3 \wedge \nu_1 \\
			\nu_1 \wedge \nu_2
		\end{bmatrix} + \frac{1}{b^2} \begin{bmatrix}
			2 + W_{11} & W_{12} & W_{13} \\
			W_{12} & 2 + W_{22} & W_{23} \\
			W_{13} & W_{23} & 2 + W_{33}
		\end{bmatrix} \begin{bmatrix}
			- \theta_1 \wedge \theta_2 + \theta_3 \wedge \theta_4 \\
			\theta_1 \wedge \theta_3 + \theta_2 \wedge \theta_4 \\
			\theta_1 \wedge \theta_4 - \theta_2 \wedge \theta_4
		\end{bmatrix}.
	\end{align}
\end{small}
\end{subequations}

\subsection{The Space of Hopf Circles} \label{sec:HopfCircle}

\indent \indent Let $M$ be a compact $7$-manifold.  In $\S$\ref{sec:3-Sasakian}, we equipped $M$ with a $3$-Sasakian structure $(\langle \cdot, \cdot\rangle$, $(\alpha_1, \alpha_2, \alpha_3)$, $(\mathsf{J}_1, \mathsf{J}_2, \mathsf{J}_3))$ and defined a ``Hopf circle" to be a leaf of the quasi-regular $1$-dimensional foliation $\mathcal{F}_w \subset \mathcal{F}_A \subset M$, for any $w \in S^2$.  We remarked that each projection $p_w \colon M \to M/\mathcal{F}_w$ is a principal $S^1$-orbibundle, and each leaf space $M/\mathcal{F}_w$ may be identified with the (orbifold) twistor space $Z$ of $X$. \\
\indent In this section, we equip $M$ with less data, namely that of a squashed quaternion-Sasakian structure $(\varphi_{a,b}, \mathsf{A})$.  In $\S$\ref{sec:HCNew}, we give two alternate characterizations of Hopf circles that are better suited to $\SO(4)$-geometry.  We also show that the space of Hopf circles may be identified with the $8$-orbifold $Z \times S^2$.  Then, in $\S$\ref{sec:almost-cpl-str}, we endow $Z \times S^2$ with a natural almost-Hermitian structure $(h_{a,b}, J, \omega_{a,b})$, and in $\S$\ref{sec:NKStr} we relate $(h_{1,1}, J, \omega_{1,1})$ to the strict nearly-K\"{a}hler structure on $Z$.

\subsubsection{Hopf Circles} \label{sec:HCNew}

\indent \indent Let $M$ be a compact $7$-manifold equipped with a squashed quaternion-Sasakian structure $(\varphi_{a,b}, \mathsf{A})$.  Recall that the distribution $\mathsf{A}$ is integrable, yielding a canonical foliation $\mathcal{F}_A$ whose leaves, called \emph{canonical leaves}, are totally-geodesic associative $3$-folds $F \subset M$ of constant positive sectional curvature.

\begin{defn}
    A \emph{Hopf circle} in $(M, \varphi_{a,b}, \mathsf{A})$ is a closed geodesic that lies in a canonical leaf.
\end{defn}

\begin{prop} \label{prop:HopfCircleSpace}
	Let $M$ be a compact squashed quaternion-Sasakian 7-manifold. The space of Hopf circles in $M$ is diffeomorphic to $Z \times S^2$, where $Z$ is the twistor space of the 4-orbifold $X = M /\mathcal{F}_A.$
\end{prop}

\begin{proof} A Hopf circle is a closed geodesic in $M$ lying entirely in a fixed canonical leaf.  Given a point $m \in M$ and a unit vector $w \in \mathsf{A}_m$, there is a unique Hopf circle passing through $m$ with tangent vector $w$. Working in a fixed $\zeta$-flat trivialization of $\mathsf{A}$, this geodesic is a fiber of the projection $p_w\colon M \to Z_w$ of Theorem \ref{thm:twistspacequot}.  Now, define an equivalence relation $\sim$ on $M \times S^2$ by declaring that $(m_1, w_1) \sim (m_2, w_2)$ if and only if $w_1 = w_2$ in this trivialization and $m_1$ and $m_2$ lie in the same fiber of $p_{w_1}\colon M \to Z_{w_1}$.  Then $(m_1, w_1)$ and $(m_2, w_2)$ define the same Hopf circle if and only if $(m_1, w_1) \sim (m_2, w_2)$.  Consequently, the space of Hopf circles is diffeomorphic to $(M \times S^2)/\!\sim\ = Z \times S^2$. 
\end{proof}

\indent We now study Hopf circles in $M$ in terms of $\SO(4)$-frames.  For this, we continue with the setup of the previous section.  That is, we let $\pi \colon \mathcal{B} \to M$ be the $\SO(4)$-coframe bundle, let $(\mu, \theta) \in \Omega^1(\mathcal{B}; \R^3 \oplus \R^4)$ denote the squashed tautological $1$-form, and let $(\zeta, \nu) \in \Omega^1(\mathcal{B}; \mathfrak{sp}(1) \oplus \mathfrak{sp}(1))$ denote the natural connection.  Thus, $(\mu,\theta,\zeta,\nu)$ is a global coframe of the $13$-manifold $\mathcal{B}$.  We now consider the following differential ideal $\mathcal{I}$ on $\mathcal{B}$ consisting of $8$ linearly independent $1$-forms
$$\mathcal{I} := \langle \theta_1, \theta_2, \theta_3, \theta_4, \mu_2, \mu_3, \zeta_2, \zeta_3 \rangle.$$
\indent By the structure equations (\ref{eq:3SasakStruct}), the ideal $\mathcal{I}$ is Frobenius.  Consequently, $\mathcal{B}$ is foliated by the $5$-dimensional integral submanifolds of $\mathcal{I}$.  Moreover, any such maximal integral $5$-fold in $\mathcal{B}$ projects via $\pi \colon \mathcal{B} \to M$ to a closed geodesic in $M$ that lies in a canonical leaf $F \subset M$ of the associative distribution $\mathsf{A}$. \\
\indent Conversely, if $\gamma\colon \left[a,b\right] \to M$ is a closed geodesic in $M$ lying entirely in a single canonical leaf $F,$ then define the subset $\mathcal{Q}(\gamma) \subset \gamma^* \mathcal{B}$ consisting of $\SO(4)$-coframes $T_mM \to \R^7$ sending the unit tangent vector of $\gamma$ to $e_1$, a $\U(2)$-subbundle of $\gamma^* \mathcal{B}$. We have $\theta_1 = \ldots = \theta_4 = 0$ and $\mu_2 = \mu_3 = 0$ on this subset by virtue of the frame adaptation, and the geodesic condition implies $\zeta_2 = \zeta_3 = 0$ on $\mathcal{Q}(\gamma)$.  Therefore, $\mathcal{Q}(\gamma)$ is a maximal 5-dimensional integral submanifold of $\mathcal{I}$.  In summary, we have proved:

\begin{prop} \label{prop:IdealDefHC} Let $M$ be a compact squashed quaternion-Sasakian 7-manifold.  A closed curve in $M$ is a Hopf circle if and only if it is the image under $\pi \colon \mathcal{B} \to M$ of a maximal integral $5$-fold of $\mathcal{I}$.
\end{prop}

\indent On the $13$-manifold $\mathcal{B}$, the vector fields dual to $\mu_1, \zeta_1, \nu_1, \nu_2, \nu_3$ generate a locally free ${\mathrm{U}(1) \times \mathrm{U}(2)}$-action, thereby yielding a principal $(\U(1) \times \U(2))$-orbibundle $\lambda \colon \mathcal{B} \to \mathcal{B}/(\U(1) \times \U(2))$.  The maximal 5-dimensional integral submanifolds of the ideal $\mathcal{I}$ are easily seen to be the orbits of ${\mathrm{U}(1) \times \mathrm{U}(2)}$ on $\mathcal{B}$.  Consequently, Propositions \ref{prop:HopfCircleSpace} and \ref{prop:IdealDefHC} imply that $\mathcal{B}/(\U(1) \times \U(2))$ can be identified with the space of Hopf circles, $Z \times S^2$.

\indent In summary, $\mathcal{B}$ admits the structure of a double fibration via the principal $\SO(4)$-bundle $\pi \colon \mathcal{B} \to M$ and the principal $(\U(1) \times \U(2))$-orbibundle $\lambda \colon \mathcal{B} \to Z \times S^2$.  Moreover, the eight $1$-forms $\beta_1, \beta_2, \beta_3, \beta_4$ and $\alpha_2, \alpha_3, \zeta_2, \zeta_3$ are $\lambda$-semibasic, whereas $\alpha_1, \zeta_1, \nu_1, \nu_2, \nu_3$ are connection $1$-forms for $\lambda$.  Altogether, we have a diagram
$$\begin{tikzcd}
& \mathcal{B} \arrow[ld, "\pi"'] \arrow[rd, "\lambda"] & \\
M \arrow[rd, "h"'] & & Z \times S^2 \arrow[ld, "q"] \\
                   & X  &                             
\end{tikzcd}$$
where $h \colon M \to X$ is the canonical fibration, and $q \colon Z \times S^2 \to X$ is the obvious $(S^2 \times S^2)$-bundle.

\subsubsection{Almost-Complex Structure on $Z \times S^2$} \label{sec:almost-cpl-str}

\indent \indent We now describe how the space $Z \times S^2= \mathcal{B}/(\U(1) \times \U(2))$ may be equipped with a naturally defined 2-parameter family of almost-Hermitian structures $(h_{a,b}, J, \omega_{a,b})$. Define $\C$-valued 1-forms $\pi_i$ on $\mathcal{B}$ by
\begin{align*}
\pi_1 & = \mu_2 + i\mu_3, & \pi_2 & = \sqrt{2}\left(\theta_1 + i\theta_2\right), & \pi_3 & = \sqrt{2}\left(\theta_3 + i\theta_4\right), & \pi_4 & = \zeta_2 + i\zeta_3.
\end{align*}
The space of $(1,0)$-forms for the almost-complex structure $J$ consists of precisely the $\C$-valued 1-forms on $Z \times S^2$ whose pullbacks to $\mathcal{B}$ lie in $\mathrm{span}_{\C}(\pi_1, \pi_2, \pi_3, \pi_4).$ This property characterizes $J.$ Note that $J$ does not depend on $a$ or $b.$

The Riemannian metric $h_{a,b}$ on $Z \times S^2$ is defined by declaring $(\pi_1, \pi_2, \pi_3, \pi_4)$ to be a $\U(4)$-frame, i.e.:
\begin{align*}
h_{a,b} & = \pi_1 \circ \overline{\pi}_1 + \pi_2 \circ \overline{\pi}_2 + \pi_3 \circ \overline{\pi}_3 + \pi_4 \circ \overline{\pi}_4 \\
& = \mu_2^2 + \mu_3^2 + \zeta_2^2 + \zeta_3^2 + 2(\theta_1^2 + \theta_2^2 + \theta_3^2 + \theta_4^2).
\end{align*}
The non-degenerate $2$-form $\omega_{a,b}$ making $(h_{a,b}, J, \omega_{a,b})$ an almost-Hermitian triple is then given by
\begin{align*}
\omega & = \textstyle \frac{i}{2}\left( \pi_1 \wedge \overline{\pi}_1 + \pi_2 \wedge \overline{\pi}_2 + \pi_3 \wedge \overline{\pi}_3 + \pi_4 \wedge \overline{\pi}_4 \right) \\
& = \mu_2 \wedge \mu_3 + \zeta_2 \wedge \zeta_3 + 2(\theta_1 \wedge \theta_2 + \theta_3 \wedge \theta_4).
\end{align*}
The structure equations (\ref{eq:3SasakStruct}) imply
\begin{align*}
        d \pi_1 & \equiv - \tfrac{a}{b^2} \, \overline{\pi_2 \wedge \pi_3} & d \pi_2 & \equiv - \tfrac{1}{a}\, \overline{\pi_3 \wedge \pi_1} - \overline{\pi_3 \wedge \pi_4}, \\
        d \pi_4 &\equiv 0 & d \pi_3 &\equiv - \tfrac{1}{a}\, \overline{\pi_1 \wedge \pi_2} + \overline{\pi_2 \wedge \pi_4},
\end{align*}
modulo $(\pi_1, \pi_2, \pi_3, \pi_4).$ Therefore, the almost-complex structure $J$ is not integrable.

\subsubsection{Strict Nearly-K\"{a}hler Structure on $Z \times \{w\}$} \label{sec:NKStr}

\indent \indent For every $w \in S^2,$ the 6-fold $Z_w = Z \times \{ w \}$ is an almost-complex sub-orbifold of $Z \times S^2.$ The almost-Hermitian structure $(h_{1,1}, J, \omega_{1,1})$ on $Z \times S^2$ restricts to $Z_w$ to define a strict nearly-K\"ahler structure: setting $a = b = 1$ and restricting to $Z_w,$ we have $\pi_4 = \overline{\pi_4} = 0$ and
\begin{equation*}
    \begin{aligned}
    d \omega &= -3 \, \mathrm{Im} \Upsilon,
    \\
    d \, \mathrm{Re} \Upsilon &= -2 \omega^2,
    \end{aligned}
\end{equation*}
where $\Upsilon = \pi_1 \wedge \pi_2 \wedge \pi_3$.  The nearly K\"ahler structure on $Z_w$ is independent of $w \in S^2$ and is equivalent to the standard nearly K\"ahler structure on $Z$ viewed as the twistor space of the positive self-dual Einstein 4-orbifold $X$.

\section{Associatives in Squashed $3$-Sasakian $7$-Manifolds} \label{sec:AssocSquash3-Sas}

\indent \indent We now begin our study of associative $3$-folds in squashed $3$-Sasakian $7$-manifolds $(M, \varphi_{a,b}, \mathsf{A})$.  In Proposition \ref{prop:stripedHopf}, we characterize the striped associatives in $M$ --- recall from Definition  \ref{def:striped} that ``striped" is a condition on the Gauss map --- as those that are locally ``Hopf-ruled."  Precisely:

\begin{defn} \label{def:HopfRuled} A $3$-dimensional submanifold $N \subset M$, not necessarily associative, is \emph{Hopf-ruled} if there exists a smooth fibration $N \to S$ to a surface $S$ whose fibers are Hopf circles in $M$.
\end{defn}

\indent In $\S$\ref{sec:HopfRuled}, we set up a formalism for Hopf ruled 3-folds (not necessarily associative) in $M$ that essentially identifies them with surfaces in $Z \times S^2$, the space of Hopf circles.  We then ask which surfaces in $Z \times S^2$ correspond to the \emph{associative} Hopf-ruled $3$-folds in $M$.  In $\S$\ref{sec:Correspondence}, we answer this question by establishing Theorem \ref{thm:MainCorrespondence}: the $J$-holomorphic curves in $(Z \times S^2, h_{a,b}, J, \omega_{a,b})$ correspond to Hopf-ruled associative $3$-folds in $(M, \varphi_{a,b})$.

\subsection{Adapted Frames for Associatives}

\indent \indent Let $(M, \varphi_{a,b}, \mathsf{A})$ be a squashed $3$-Sasakian $7$-manifold.  In this section, we study associative $3$-folds in $M$ with respect to the $\SO(4)$-structure $(\varphi_{a,b}, \mathsf{A})$ by the method of moving frames.  We begin by adapting $\SO(4)$-frames to the associative. \\
\indent Let $\mathcal{B} \to M$ be the bundle of $\SO(4)$-coframes as in $\S$\ref{sec:SO(4)-Coframe}, and let $u \colon N \to M$ be an associative $3$-fold.  At each point $p \in N$, the tangent space $T_pN$ is an element of the associative Grassmannian $\Gr_{\text{ass}}(T_pM, \varphi_{a,b}) \cong \G_2/\SO(4)$.  By Proposition \ref{prop:assocorbits}, there exists a coframe $\left(\mu_1, \mu_2, \mu_3, \theta_1, \theta_2, \theta_3, \theta_4 \right) \in \mathcal{B}_p$ and numbers $r \in [0,\pi/2]$ and $s \in [0,\pi/6]$ such that $T_p N$ is defined by the equations
 \begin{equation}\label{eq:frameadaptzeros}
	\begin{aligned}
		-\sin (2s) \mu_1 &+ \cos (2s) \theta_3 = 0, \\
		-\sin(s+r) \mu_2 &+ \cos(s+r) \theta_2 = 0, \\
		-\sin(s-r) \mu_3 &+ \cos(s-r) \theta_1 = 0, \\
		\theta_4 &= 0.
	\end{aligned}
 \end{equation}
As $p \in N$ varies, we obtain a pair of functions $r \colon N \to [0,\pi/2]$ and $s \colon N \to [0,\pi/6]$ together with a subset $\mathcal{B}_1$ of the pulled back bundle $u^*(\mathcal{B})$ consisting of all coframes satisfying (\ref{eq:frameadaptzeros}).

If we assume that the tangent spaces of the associative have uniform $\mathrm{SO}(4)$-stabilizer $\G$, then $\mathcal{B}_1$ defines a principal $\G$-bundle over $N$ and we have a diagram
 $$\begin{tikzcd}
	{\mathrm{G}} & {\mathrm{SO}(4)} & {\mathrm{SO}(4)} \\
	& {\mathcal{B}_1} & {u^*(\mathcal{B})} & {\mathcal{B}} \\
	& N & N & M.
	\arrow[from=1-2, to=2-3]
	\arrow[from=1-3, to=2-4]
	\arrow[from=2-3, to=2-4]
	\arrow[from=2-3, to=3-3]
	\arrow[from=2-4, to=3-4]
	\arrow[from=3-3, to=3-4]
	\arrow[from=1-1, to=2-2]
	\arrow[hook, from=2-2, to=2-3]
	\arrow[no head, from=3-2, to=3-3]
	\arrow[from=2-2, to=3-2]
\end{tikzcd}$$
For example, if $u \colon N \to M$ is striped and $r \in (0, \pi/2)$, then the tangent spaces all have stabilizer $\mathrm{G} \cong \text{O}(2) \leq \SO(4)$ for a suitable embedding of $\text{O}(2) \hookrightarrow \SO(4)$, and $\mathcal{B}_1 \to N$ is a principal $\text{O}(2)$-bundle.

\subsection{Striped Associative $3$-folds}

\indent \indent Let $u \colon N \to M$ be a \emph{striped} associative submanifold with adapted coframe bundle $\mathcal{B}_1 \to N$. By Definition \ref{def:striped}, we have $s = 0$ identically on $\mathcal{B}_1$, so that
\begin{align*}
    -\sin(r) \mu_2 + \cos(r) \theta_2 & = 0, & \theta_3 & = 0, \\
	\sin(r) \mu_3 + \cos(r) \theta_1 & = 0, & \theta_4 &= 0.
\end{align*}
Let us write $\cos(r) = x_1$ and $\sin(r) = x_2,$ so that $x_1^2 + x_2^2 = 1,$ and define 1-forms $\epsilon_1, \epsilon_2, \epsilon_3$ on $\mathcal{B}_1$ by
\begin{equation*}
	\begin{aligned}
		\epsilon_1 &= \mu_1 \\
		\epsilon_2 &= x_1 \mu_2 - x_2 \theta_2, \\
		\epsilon_3 &= x_1 \mu_3 + x_2 \theta_1.
	\end{aligned}
\end{equation*}
The $1$-forms $\epsilon_1, \epsilon_2, \epsilon_3$ are semi-basic for the projection $\mathcal{B}_1 \to N$.   The metric and volume form on $N$ induced from $M$ pull back to $\mathcal{B}_1$ to be
\begin{equation*}
	g_N = \epsilon_1^2 + \epsilon_2^2 + \epsilon_3^2, \:\:\:\: \mathrm{vol}_N = \epsilon_{123}.
\end{equation*}
The equations (\ref{eq:frameadaptzeros}) together with the structure equations (\ref{eq:Scaled3SasakStruct}) imply
\begin{subequations}\label{eq:TypeAstruct}
\begin{equation}\label{eq:TypeAstructa}
	\begin{aligned}
		d \epsilon_1 &= \left(-4 f_1 x_1^2 + 2 \frac{a}{b^2} x_1^2 - 4 \frac{1}{a} x_1^2 - 2 \frac{a}{b^2} \right) \epsilon_2 \wedge \epsilon_3, & & \\
		d \epsilon_2 &= - \left(\psi - \frac{1}{a} \epsilon_1 \right) \wedge \epsilon_3 - 2 f_2 x_1^2 \epsilon_2 \wedge \epsilon_3 + \frac{x_1}{x_2} y \, \epsilon_1 \wedge \epsilon_2, & & \\
		d \epsilon_3 &= \left(\psi - \frac{1}{a} \epsilon_1 \right) \wedge \epsilon_2 + 2 f_3 x_1^2 \epsilon_3 \wedge \epsilon_2 + \frac{x_1}{x_2} y \, \epsilon_1 \wedge \epsilon_3, & & \\
		d x_1 &= x_2 y \, \epsilon_1 - 2 x_1 x_2^2 f_3 \, \epsilon_2 + 2 x_1 x_2^2 f_2 \, \epsilon_3, & & \\
		d x_2 &= -x_1 y \, \epsilon_1 + 2 x_1^2 x_2 f_3 \, \epsilon_2 - 2 x_1^2 x_2 f_2 \, \epsilon_3,
	\end{aligned}
\end{equation}
and
\begin{equation}\label{eq:TypeAstructb}
    \begin{aligned}
        \zeta_1 &= \tfrac{1}{2} \psi + f_1 \epsilon_1 + f_2 \epsilon_2 + f_3 \epsilon_3, & \nu_1 &= -\tfrac{3}{2} \psi -f_1 \epsilon_1 - f_2 \epsilon_2 - f_3 \epsilon_3. \\
		\zeta_2 &= - \frac{x_1}{2 a} \left(2 a f_1 + 3 \right) \epsilon_2 + \frac{y}{2 x_2} \epsilon_3, & \nu_2 &= p_2 \epsilon_2 + p_3 \epsilon_3,  \\
		\zeta_3 &= -\frac{y}{2 x_2} \epsilon_2 - \frac{x_1}{2 a} \left(2 a f_1 + 3 \right) \epsilon_3 , & \nu_2 &= -p_3 \epsilon_2 + p_2 \epsilon_3,
    \end{aligned}
\end{equation}
\end{subequations}
where $y,$ $f_1, f_2, f_3,$ $p_2, p_3$ are functions defined on $\mathcal{B}_1$ and $\psi$ is a connection 1-form for the $\mathrm{S}^1$-structure $\mathcal{B}_1$ on $N.$

\begin{prop} \label{prop:stripedHopf} Let $u \colon N \to (M, \varphi_{a,b})$ be an associative $3$-fold.  Then:
$$N \text{ is striped or locally a subset of a canonical leaf } \iff N \text{ is locally ruled by Hopf circles.}$$
\end{prop}

\begin{proof} $(\Longleftarrow)$ Suppose $N$ is locally ruled by Hopf circles.  Then for each $p \in N$, we have $\dim(T_pN \cap \mathsf{A}_p) \geq 1$.  The result now follows from Proposition \ref{prop:AssocIntersect} and Definition \ref{def:striped}. \\ 
\indent $(\Longrightarrow)$ If $N$ is a subset of a canonical leaf, then the result is clear.  So, suppose that $u \colon N \to M$ is striped. Equations (\ref{eq:TypeAstruct}) show that the ideal $\langle \epsilon_2, \epsilon_3 \rangle$ defined on the $4$-manifold $\mathcal{B}_1$ is Frobenius.  They also imply that the integral surfaces of $\langle \epsilon_2, \epsilon_3 \rangle$ on $\mathcal{B}_1$ are integral surfaces of the ideal $\left\langle \mu_2, \mu_3, \theta_1, \theta_2, \theta_3, \theta_4, \zeta_2, \zeta_3 \right\rangle$ defined on $u^*(\mathcal{B})$.  Therefore, these integral surfaces project to Hopf circles that lie in $N$.
\end{proof}
\begin{rmk}
We emphasize that the definition of ``striped associative" makes sense in any $7$-manifold with an $\SO(4)$-structure $(\varphi, \mathsf{A})$.  On the other hand, the concept of ``Hopf-ruled $3$-fold" is specific to \emph{quaternion-Sasakian} $\SO(4)$-structures and their squashings. It would be interesting to know if striped associatives in $7$-manifolds with other types of $\SO(4)$-structures admit similar geometric characterizations. 
\end{rmk}

\subsection{Hopf-Ruled $3$-folds} \label{sec:HopfRuled}

\indent \indent Let $(M, \varphi_{a,b}, \mathsf{A})$ be a compact squashed $3$-Sasakian $7$-manifold.  Recall that Proposition \ref{prop:HopfCircleSpace} gives a diffeomorphism
\begin{align*}
    \Gamma \colon Z \times S^2 & \to \frac{\mathcal{B}}{\U(1) \times \U(2)} \cong \{\text{Hopf circles in }M\} \\
    \Gamma(z,w) & = \pi(\lambda^{-1}(z,w)),
\end{align*}
where we are viewing $\mathcal{B}$ as a double fibration
$$\begin{tikzcd}
& \mathcal{B} \arrow[ld, "\pi"'] \arrow[rd, "\lambda"] &                                              \\
M & & Z \times S^2 
\end{tikzcd}$$
Heuristically, this suggests correspondences between the following classes of objects:
\begin{enumerate}
    \item A Hopf-ruled $3$-fold in $M$ --- i.e., an immersed $3$-fold $u \colon N \to M$ admitting a fibration to a surface whose fibers are Hopf circles.
    \item An immersed $3$-fold in $M$ of the form $\Gamma(\Sigma) := \pi(\lambda^{-1}(\Sigma)) \subset M$ for some surface $\Sigma \subset Z \times S^2$.
    \item A $2$-parameter family of Hopf circles --- i.e., an immersed surface $v \colon \Sigma \to Z \times S^2$.
    \item A ``directrix" surface $f \colon \Sigma \to Z$ together with a ``ruling" map $w \colon \Sigma \to S^2$ for which $\text{rank}(f,w) = 2$.
\end{enumerate}
Guided by this intuitive picture, in this section we explain how to recast the geometry of a Hopf-ruled $3$-fold $u \colon N \to M$ in terms of an associated surface $v \colon \Sigma \to Z \times S^2$, and vice versa.  The formalism that we establish here will be used to prove the Correspondence Theorem \ref{thm:MainCorrespondence}. 

\subsubsection{From Surfaces in $Z \times S^2$ to Hopf-Ruled $3$-folds in $M$} \label{sec:Surface2Ruled}

\indent \indent Let $v \colon \Sigma \to Z \times S^2$ be an immersed surface.  Recall that $Z \times S^2$ carries the Riemannian metric $h_{a,b}$ defined in $\S$\ref{sec:almost-cpl-str}.
Equip $\Sigma$ with the metric $h_\Sigma$ induced from $v \colon \Sigma \to (Z \times S^2, h_{a,b})$, let $\mathcal{F}_\Sigma \to \Sigma$ be the corresponding oriented orthonormal frame bundle, and let $\eta = \eta_1 + i\eta_2 \in \Omega^1(\mathcal{F}_\Sigma; \C)$ be the complex tautological form.  Thus, the induced metric and volume form on $\Sigma$ are
\begin{align*}
    h_\Sigma & = \eta \circ \overline{\eta} = \eta_1^2 + \eta_2^2 & \text{vol}_\Sigma & = \textstyle \frac{i}{2}\eta \wedge \overline{\eta} = \eta_1 \wedge \eta_2.    
\end{align*}
Next, noting that $v^*(\mathcal{B}) \to \Sigma$ is a $(\U(1) \times \U(2))$-bundle, we may define the $(T^2 \times \U(2))$-bundle
$$\mathcal{G}_\Sigma := \mathcal{F}_\Sigma \times_\Sigma v^*(\mathcal{B}).$$
We have a diagram
$$\begin{tikzcd}
\mathcal{G}_\Sigma \arrow[r, two heads] \arrow[d] & v^*(\mathcal{B}) \arrow[d] \arrow[r, hook] & \mathcal{B} \arrow[d, "\lambda"] \arrow[rd, "\pi"] &              \\
\Sigma \arrow[r, no head]                         & \Sigma \arrow[r, "v"]                        & Z \times S^2                               & M
\end{tikzcd}$$
Now, note that the image of $v^*(\mathcal{B}) \hookrightarrow \mathcal{B}$ is simply $\lambda^{-1}(\Sigma)$.  Therefore, the image of the natural map $\mathcal{G}_\Sigma \to M$ is equal to $(\pi \circ \lambda^{-1})(\Sigma)$.  Moreover, the image of $\mathcal{G}_\Sigma \to M$ is a Hopf-ruled $3$-dimensional immersed submanifold on the set
$$\mathcal{G}_\Sigma^{\circ} := \{ p \in \mathcal{G}_\Sigma \colon \text{rank}(\mu_1, \mu_2, \mu_3, \theta_1, \theta_2, \theta_3, \theta_4)|_p = 3 \},$$
i.e., the set of points on which the linear map $(\mu_1, \mu_2, \mu_3, \theta_1, \theta_2, \theta_3, \theta_4) \colon T_p\mathcal{G}_\Sigma \to \R^7$ has rank $3$.

\subsubsection{From Hopf-Ruled $3$-folds in $M$ to Surfaces in $Z \times S^2$} \label{sec:Ruled2Surface}

\indent \indent Let $u \colon N \to M$ be a Hopf-ruled immersed $3$-fold in $M$, and let $V$ be the unit vector field on $N$ tangent to the ruling.  Consider $u^*(\mathcal{B}) \to N$, which is a $\SO(4)$-bundle.  Adapt frames to the ruling by considering
$$\mathcal{B}^{1}_N|_p  := \{ (e_1, \ldots, e_7) \in u^*(\mathcal{B})|_p \colon e_1 = V_p \}.$$
Note that $\mathcal{B}^{1}_N \to N$ is a $\U(2)$-bundle. \\
 \indent Next, equip $N$ with the Riemannian metric $g_N$ induced from $u \colon N \to M$, let $\mathcal{F}_N \to N$ be the $\SO(3)$-coframe bundle, and let $\rho = (\rho_1, \rho_2, \rho_3) \in \Omega^1(\mathcal{F}_N; \R^3)$ be the tautological $1$-form.
Again, we can adapt frames to the ruling by considering the $\SO(2)$-bundle
$$\mathcal{F}^{1}_N := \{ f \in \mathcal{F}_N \colon f_*(V) = (1,0,0) \}.$$
Finally, define the $(\SO(2) \times \U(2))$-bundle
$$\mathcal{Q}_N := \mathcal{F}^1_N \times_N \mathcal{B}^1_N.$$
Altogether, we have a diagram
$$\begin{tikzcd}
\mathcal{Q}_N \arrow[r, two heads] \arrow[d] & \mathcal{B}^1_N \arrow[r, hook] \arrow[d] & u^*(\mathcal{B}) \arrow[d] \arrow[r, hook] & \mathcal{B} \arrow[d, "\pi"] \arrow[rd, "\lambda"] &                                   \\
N \arrow[r, no head]              & N \arrow[r, no head]                & N \arrow[r, "u"]        & M      & Z \times S^2
\end{tikzcd}$$
One can check that the image of the natural map $\mathcal{Q}_N \to \mathcal{B}$ is equal to $\lambda^{-1}(\Sigma)$ for some immersed surface $\Sigma \to Z \times S^2$, and hence $N = (\pi \circ \lambda^{-1})(\Sigma)$. \\

\indent The above formalism established, we now explore the geometry of a Hopf-ruled $3$-fold $u \colon N \to M$.  For this, recall the $\C$-valued $1$-forms $\pi_1, \pi_2, \pi_3, \pi_4 \in \Omega^1(\mathcal{B}; \C)$ discussed in $\S$\ref{sec:almost-cpl-str}.  In terms of these, the $\G_2$-structure and metric on $M^7$ are:
\begin{equation}\label{eq:piexpressphig}
\begin{aligned}
\varphi_{a,b} & = \textstyle \mu_1 \wedge \frac{i}{2} (\pi_1 \wedge \overline{\pi}_1 + \frac{1}{2}\pi_2 \wedge \overline{\pi}_2 + \frac{1}{2}\pi_3 \wedge \overline{\pi}_3) + \frac{1}{2}\,\text{Re}(\pi_1 \wedge \pi_2 \wedge \pi_3) \\
g_{a,b} & \textstyle = \mu_1^2 + \pi_1 \circ \overline{\pi}_1 + \frac{1}{2}\pi_2 \circ \overline{\pi}_2 + \frac{1}{2}\pi_3 \circ \overline{\pi}_3.
\end{aligned}
\end{equation}
Pulling all quantities back to $\mathcal{Q}_N$, we have:

\begin{lem} \label{lem:geometry-Hopf-ruled} Let $\sigma = \rho_2 + i\rho_3$.  There exist $\C$-valued functions $A_1, \ldots, A_4, B_1, \ldots, B_4 \colon \mathcal{Q}_N \to \C$ for which
\begin{align*}
\pi_1 & = A_1\sigma + B_1\overline{\sigma} &  \pi_2 & = A_2\sigma + B_2\overline{\sigma} \\
\pi_4 & = A_4\sigma + B_4\overline{\sigma} & \pi_3  & = A_3\sigma + B_3\overline{\sigma}.
\end{align*}
Write $A = (A_1, \frac{1}{\sqrt{2}}A_2, \frac{1}{\sqrt{2}}A_3)$ and $B = (B_1, \frac{1}{\sqrt{2}}B_2, \frac{1}{\sqrt{2}}B_3)$, viewed as vectors in $\C^3$, and let $\langle A, B \rangle := A_1\overline{B}_1 + \frac{1}{2}A_2\overline{B}_2 + \frac{1}{2}A_3\overline{B}_3$ denote their inner product.  Then on $\mathcal{Q}_N$,
\begin{align*}
\left.\varphi_{a,b}\right|_N & = \frac{i}{2}\left( \Vert A \Vert^2 - \Vert B \Vert^2 \right) \mu_1 \wedge \sigma \wedge \overline{\sigma} \\
g_{a,b \, N} & = \mu_1^2 + \left( \Vert A \Vert^2 + \Vert B \Vert^2 \right)\sigma \circ \overline{\sigma} + 2\,\mathrm{Re}\!\left[ \langle A,B \rangle\, \sigma^2 \right]  \\
\vol_{a,b \, N} & = \frac{i}{2}\left[ \left( \Vert A \Vert^2 + \Vert B \Vert^2 \right)^2 - 4 \left| \langle A, B \rangle \right|^2 \right]^{1/2} \mu_1 \wedge \sigma \wedge \overline{\sigma}.
\end{align*}
\end{lem}

\begin{proof} 
By the definitions of $\mathcal{B}^1_N$ and $\mathcal{F}^1_N$ we have $\mu_1 = \rho_1$ on $\mathcal{Q}_N.$ On the one hand, the Hopf circles in $N$ are the projections to $N$ of the integral submanifolds of the ideal $\langle \rho_2, \rho_3 \rangle = \langle \sigma, \overline{\sigma} \rangle.$ On the other hand, Hopf circles are the projections to $N$ of the integral submanifolds of the ideal $\left\langle \mu_2, \mu_3, \theta_1, \theta_2, \theta_3, \theta_4, \zeta_2, \zeta_3 \right\rangle.$ Therefore, on $\mathcal{Q}_N,$ we must have
\begin{align*}
\mu_2, \mu_3, \zeta_2, \zeta_3, \theta_1, \theta_2, \theta_3, \theta_4 & \equiv 0\,\text{ mod }(\sigma, \overline{\sigma}).
\end{align*}
This yields the existence of $A_1, A_2, A_3, A_4, B_1, B_2, B_3, B_4$.  The remaining formulas follow by using (\ref{eq:piexpressphig}) to express $\varphi_{a,b}|_N$ and $g_{a,b \, N}$ in terms of these functions.
\end{proof}

\subsection{Hopf-Ruled Associatives: Correspondence Theorem} \label{sec:Correspondence}

\indent \indent We are now in a position to prove the Correspondence Theorem \ref{thm:MainCorrespondence}. Recall that $Z \times S^2$ carries the almost-Hermitian structure $(h_{a,b}, J, \omega_{a,b})$ defined in $\S$\ref{sec:almost-cpl-str}.  An immersed surface $v \colon \Sigma \to Z \times S^2$ is called \emph{$J$-holomorphic} if each tangent plane $v_*(T_p\Sigma) \subset T_{v(p)}(Z \times S^2)$ is $J$-invariant.  In practice, we will use the following alternate characterizations of $J$-invariant $2$-planes.

\begin{lem}[Characterizations of complex lines] \label{lem:holo-char} Let $(V, (h,J, \omega))$ be a Hermitian vector space, and let $E \subset V$ be an oriented real $2$-dimensional subspace.  The induced metric and orientation on $E$ endow it with a complex structure (that is not necessarily the restriction of $J$). \\
\indent (a) $E \subset V$ is $J$-invariant $\iff$ Every $\beta \in \Lambda^{2,0}(V^*)$ satisfies $\beta|_E = 0$. \\
\indent (b) $E \subset V$ is $J$-invariant and positively oriented $\iff$ Every $\alpha \in \Lambda^{1,0}(V^*)$ satisfies $\alpha|_E \in \Lambda^{1,0}(E^*)$.
\end{lem}

Finally, recall that $Z \times S^2$ has the structure of an $(S^2 \times S^2)$-bundle $q \colon Z \times S^2 \to X$.  We can now prove: 

\begin{thm}[Correspondence Theorem] \label{thm:Correspondence} Let $(M^7, \varphi_{a,b})$ be a compact squashed $3$-Sasakian $7$-manifold, where $a,b > 0$ are fixed.
\begin{enumerate}[label=(\alph*)]
\item Every Hopf-ruled associative $3$-fold in $(M^7, \varphi_{a,b})$ is locally of the form $\Gamma(\Sigma)$ for some $J$-holomorphic curve $\Sigma \subset Z \times S^2$.
\item Let $\Sigma \subset Z \times S^2$ be a $J$-holomorphic curve.
\begin{enumerate}[label=\arabic*.]
    \item If $\Sigma$ does not lie in an $(S^2 \times S^2)$-fiber, then there exists a discrete set $D \subset \Sigma$ such that $\Gamma(\Sigma - D) \subset M$ is a Hopf-ruled associative $3$-fold. 
    \item If $\Sigma$ lies in an $(S^2 \times S^2)$-fiber, then $\Gamma(\Sigma)$ is a subset of a canonical leaf.
\end{enumerate}
\end{enumerate}
\end{thm}
\begin{proof} (a) Suppose that $u \colon N^3 \to M^7$ is an immersed Hopf-ruled associative $3$-fold.  Since $u$ is a Hopf-ruled $3$-fold, we may apply the formalism of $\S$\ref{sec:Ruled2Surface} and perform calculations on the space $\mathcal{Q}_N$.  Recall also that there exists a immersed surface $v \colon \Sigma \to Z \times S^2$ for which $\lambda^{-1}(\Sigma) \subset \mathcal{B}$ is the image of the natural map $\mathcal{Q}_N \to \mathcal{B}$. \\
\indent Now, since $u$ is associative, we have $\left.\varphi_{a,b}\right|_N = \vol_{a,b\,N}$.  Therefore, Lemma \ref{lem:geometry-Hopf-ruled} yields
$$\left( \Vert A \Vert^2 - \Vert B \Vert^2 \right)^2 = \left( \Vert A \Vert^2 + \Vert B \Vert^2 \right)^2 - 4 \left| \langle A,B \rangle \right|^2.$$
Expanding and simplifying yields $\Vert A \Vert \Vert B \Vert =  \left| \langle A,B \rangle\right|$, and so $\{A,B\}$ is a $\C$-linearly dependent set in $\C^3$.  It follows that
\begin{align*}
\pi_1 \wedge \pi_2 & = 0 & \pi_1 \wedge \pi_3 & = 0 & \pi_2 \wedge \pi_3 & = 0.
\end{align*}
In particular, $d(\pi_1 \wedge \pi_2) = 0$ and $d(\pi_1 \wedge \pi_3) = 0$.  Expanding these via the structure equations (\ref{eq:3SasakStruct}), 
we obtain $A_2B_4 - A_4B_2 = 0$ and $A_3B_4 - A_4B_3 = 0$, whence $\pi_2 \wedge \pi_4 = 0$ and $\pi_3 \wedge \pi_4 = 0$.  Considering the vectors $(A_1, A_2, A_3, A_4)$ and $(B_1, B_2, B_3, B_4)$ in $\C^4$, a linear algebra argument shows that $A_1B_4 - A_4B_1 = 0$, so that $\pi_1 \wedge \pi_4 = 0$. \\
\indent Altogether, we have shown that $\pi_i \wedge \pi_j = 0$ for all $i,j = 1, \ldots, 4$.  By Lemma \ref{lem:holo-char}(a), we conclude that $v \colon \Sigma \to Z \times S^2$ is a $J$-holomorphic curve. \\

\indent (b) Let $v \colon \Sigma \to Z \times S^2$ be an immersed $J$-holomorphic curve.  Since $v 
\colon \Sigma \to Z \times S^2$ is an immersed surface, we may apply the formalism of $\S$\ref{sec:Surface2Ruled} and perform calculations on the space $\mathcal{G}_\Sigma$.  Since $v$ is $J$-holomorphic, Lemma \ref{lem:holo-char}(b) implies that there exist functions $W_1$, $W_2$, $W_3$, $W_4 \colon \mathcal{G}_\Sigma \to \C$ for which
\begin{align*}
\pi_1 & = W_1 \eta & \pi_2 & = W_2 \eta & \pi_3 & = W_3 \eta & \pi_4 & = W_4 \eta.
\end{align*}
and since $\omega_{a,b}|_\Sigma = \vol_\Sigma$,
\begin{equation}\label{eq:Wrelation}
    |W_1|^2+|W_2|^2+|W_3|^2+|W_4|^2 = 1
\end{equation}
The proof will consist of five steps: (i) Show that the image of $\mathcal{G}_\Sigma^{\circ} \to M^7$ is associative; (ii) Show that $\mathcal{G}_\Sigma - \mathcal{G}_\Sigma^\circ = \{W_1 = W_2 = W_3 = 0\}$; (iii) Compute the derivatives of $W_1, \ldots, W_4$; (iv) Exhibit $\mathcal{Y} := \{W_2 = W_3 = 0\}$ as the zero set of a holomorphic section over $\Sigma$; (v) Complete the proof. \\

\indent (i) Let $N \subset M$ be the image of $\mathcal{G}_\Sigma^\circ \to M$.  Recall that the $\G_2$-structure and metric on $M^7$ are
\begin{align*}
\varphi_{a,b} & = \textstyle \mu_1 \wedge \frac{i}{2} (\pi_1 \wedge \overline{\pi}_1 + \frac{1}{2}\pi_2 \wedge \overline{\pi}_2 + \frac{1}{2}\pi_3 \wedge \overline{\pi}_3) + \frac{1}{2}\text{Re}(\pi_1 \wedge \pi_2 \wedge \pi_3) \\
g_{a,b} & = \textstyle \mu_1^2 + \pi_1 \circ \overline{\pi}_1 + \frac{1}{2}\pi_2 \circ \overline{\pi}_2 + \frac{1}{2}\pi_3 \circ \overline{\pi}_3
\end{align*}
Therefore, on $\mathcal{G}_\Sigma$:
\begin{align*}
\varphi_{a,b} & = \textstyle \left( |W_1|^2 + \frac{1}{2}|W_2|^2 + \frac{1}{2}|W_3|^2 \right) \frac{i}{2}\mu_1 \wedge \eta \wedge \overline{\eta} \\
g_{a,b}|_{\mathcal{G}_\Sigma} & = \textstyle \mu_1^2 + \left( |W_1|^2 + \frac{1}{2}|W_2|^2 + \frac{1}{2}|W_3|^2 \right) \eta \circ \overline{\eta}
\end{align*}
The volume form on $N$ induced from $g_{a,b}|_{\mathcal{G}_\Sigma}$ is then
$$\textstyle \text{vol}_N = \left( |W_1|^2 + \frac{1}{2}|W_2|^2 + \frac{1}{2}|W_3|^2 \right)\frac{i}{2}\alpha_1 \wedge \eta \wedge \overline{\eta}.$$
This shows that the image of $\mathcal{G}_\Sigma^\circ \to M$ is associative.  To complete the proof, we need to construct $\Sigma^\circ := \Sigma - D$.  We will do this by identifying an open dense set $\Sigma^\circ \subset \Sigma$ with the property that $\mathcal{G}_{\Sigma^{\circ}} \subset \mathcal{G}_\Sigma^\circ$. \\

\indent (ii) On $\mathcal{G}_\Sigma$, we have
\begin{align*}
\mu_2 \wedge \mu_3 & = \frac{i}{2}|W_1|^2\,\eta \wedge \overline{\eta} & \theta_1 \wedge \theta_2 & = \frac{i}{4}|W_2|^2\,\eta \wedge \overline{\eta} & \theta_3 \wedge \theta_4 & = \frac{i}{4}|W_3|^2\,\eta \wedge \overline{\eta}
\end{align*}
and
\begin{align*}
\theta_2 \wedge \theta_4 & = \theta_1 \wedge \theta_3 = \textstyle \frac{1}{8}\left(W_2\overline{W}_3 - \overline{W}_2W_3\right)\eta \wedge \overline{\eta}  \\
\theta_1 \wedge \theta_4 & = - \theta_2 \wedge \theta_3 = \textstyle \frac{i}{8}\left(W_2 \overline{W}_3 + \overline{W}_2W_3\right)\eta \wedge \overline{\eta}  \\
\mu_2 \wedge \theta_1 & = \mu_3 \wedge \theta_2 = \textstyle \frac{1}{4\sqrt{2}} \left(W_1 \overline{W}_2 - \overline{W}_1W_2\right)\eta \wedge \overline{\eta} \\
\mu_2 \wedge \theta_2 & = -\mu_3 \wedge \theta_1 =  \textstyle \frac{i}{4\sqrt{2}}\left(W_1\overline{W}_2 + \overline{W}_1W_2\right)\eta \wedge \overline{\eta} \\
\mu_2 \wedge \theta_3 & = \mu_3 \wedge \theta_4 = \textstyle \frac{1}{4\sqrt{2}} \left(W_1\overline{W}_3 - \overline{W}_1W_3
\right)\eta \wedge \overline{\eta} \\
\mu_2 \wedge \theta_4 & = - \mu_3 \wedge \theta_3 = \textstyle \frac{i}{4\sqrt{2}} \left(W_1\overline{W}_3 + \overline{W}_1W_3\right)\eta \wedge \overline{\eta}
\end{align*}
It follows that $\mathcal{G}_\Sigma - \mathcal{G}_\Sigma^\circ = \{W_1 = W_2 = W_3 = 0\}$. \\ 

\indent (iii) Let $\psi$ denote the connection 1-form for the Levi-Civita connection on $\Sigma$ associated to the metric $h_\Sigma,$ so that $\eta$ satisfies the the structure equation
\begin{equation*}
    d \eta = i \, \psi \wedge \eta.
\end{equation*}
Now, for each $j = 1,2,3,4$, we have
$$dW_j \wedge \eta = d(W_j\eta) - W_j\,d\eta = d\pi_j - iW_j \psi \wedge \eta.$$
Calculating $d\pi_j$ via the structure equations (\ref{eq:3SasakStruct}), we find
\begin{align*}
d(W_1 - i a W_4) \wedge \eta & = 2i(W_1 - i a W_4) \alpha_1 \wedge \eta + 2 i \left( W_1 - i a W_4 \right) \zeta_1 \wedge \eta - i(W_1 - i a W_4) \psi \wedge \eta \\
dW_2 \wedge \eta & = - iW_2  \alpha_1 \wedge \eta - i W_2 \zeta_1 \wedge \eta - iW_2  \nu_1 \wedge \eta - W_3(  \nu_2 - i\nu_3) \wedge \eta  - iW_2 \psi \wedge \eta \\
dW_3 \wedge \eta & = -i W_3 \alpha_1 \wedge \eta - iW_3  \zeta_1 \wedge \eta + iW_3  \nu_1 \wedge \eta - W_2( - \nu_2 - i\nu_3 )\wedge \eta - iW_3 \psi \wedge \eta \\
d W_4 \wedge \eta & = 2i W_4 \zeta_1 \wedge \eta - i W_4 \psi \wedge \eta
\end{align*}
Set $X := W_1 - i a W_4$. By Cartan's Lemma, there exist functions $F_1, F_2, F_3, F_4 \colon \mathcal{G}_\Sigma \to \C$ for which
\begin{equation} \label{eq:WDerivatives}
\begin{aligned}
d X  & = 2i X (\alpha_1 + \zeta_1)  - iX \psi + F_1 \eta, \\
d W_4 & = 2i W_4 \zeta_1  - i W_4 \psi  + F_4 \eta, \\
\begin{bmatrix} dW_2 \\ dW_3 \end{bmatrix} & = - \begin{pmatrix} i\nu_1 + i \zeta_1 & -\nu_2 - i\nu_3 \\ \nu_2 - i\nu_3 & -i\nu_1 + i \zeta_1 \end{pmatrix} \begin{bmatrix} W_2 \\ W_3 \end{bmatrix} + \begin{pmatrix} -i\alpha_1 - i\psi & 0 \\ 0 & -i\alpha_1 - i \psi \end{pmatrix} \begin{bmatrix} W_2 \\ W_3 \end{bmatrix} + \begin{bmatrix} F_2 \\ F_3 \end{bmatrix}\eta.
\end{aligned}
\end{equation}
\indent (iv) The first two equations of (\ref{eq:WDerivatives}) imply that $|X|^2$ and $|W_4|^2$ descend to well-defined functions on $\Sigma$. Similarly, the last equation of (\ref{eq:WDerivatives}) implies that $|W_2|^2 + |W_3|^2$ descends to a well-defined function on $\Sigma$.  We may thus consider the following set:
\begin{align*}
\mathcal{Y} & := \{W_2 = W_3 = 0\} \subset \Sigma
\end{align*}
\indent Let $G = \mathrm{T}^2 \times \U(2),$ the structure group of the bundle $\mathcal{G}_\Sigma \to \Sigma$. The vector fields dual to $\alpha_1$ and $\psi$ generate the $\mathrm{T}^2$-action on $\mathcal{G}_\Sigma$ and the vector fields dual to $\zeta_1, \nu_1, \nu_2, \nu_3$ generate the $\mathrm{U}(2)$-action on $\mathcal{G}_\Sigma$.  The function $X \colon \mathcal{G}_\Sigma \to \C$ is $G$-equivariant with respect to the $G$-representation on $\C$ for which $\mathrm{U}(2)$ acts via the determinant representation and $\mathrm{T}^2$ acts with weights $(2,-1)$ and the function $W_4 \colon \mathcal{G}_\Sigma \to \C$ is $G$-equivariant with respect to the $G$-representation on $\C$ for which $\mathrm{U}(2)$ acts via the determinant representation and $\mathrm{T}^2$ acts with weights $(0,-1)$. Therefore, $X$ and $W_4$ correspond to sections $\Phi_{\mathcal{X}}, \Phi_{\mathcal{W}_4}$ of associated complex line bundles $\mathscr{L}_\mathcal{X}, \mathscr{L}_{\mathcal{W}_4}$ over $\Sigma$.   Similarly, $(W_2, W_3) \colon \mathcal{G}_\Sigma \to \C$ is $G$-equivariant with respect to the $G$-representation on $\C^2$ given by the tensor product of the standard $\U(2)$-representation with the $\mathrm{T}^2$-representation having weights $(-1,-1)$ and so corresponds to a section $\Phi_{\mathcal{H}}$ of an associated rank-$2$ complex vector bundle $\mathscr{L}_{\mathcal{H}} \to \Sigma$. \\
\indent Let us equip all of these vector bundles with their Koszul-Malgrange holomorphic structure with respect to the connection on the principal $G$-bundle $\mathcal{G}_\Sigma \to \Sigma$ whose connection forms are $\left( \alpha_1, \psi, \zeta_1, \nu_1, \nu_2, \nu_3 \right)$.  Equations (\ref{eq:WDerivatives}) show that the sections $\Phi_{\mathcal{X}}, \Phi_{\mathcal{Y}}, \Phi_{\mathcal{H}}$ are holomorphic.  Moreover, since
\begin{align*}
\mathcal{Y} & = \{ \Phi_{\mathcal{H}} = 0 \}
\end{align*}
we see that $\mathcal{Y}$ is the vanishing locus of a holomorphic section.  In particular, $\mathcal{Y}$ is discrete or $\mathcal{Y} = \Sigma$. \\

\indent (v) If $\mathcal{Y}$ is discrete, then we are done: the set $D$ is simply $\mathcal{Y}$. We therefore assume for the rest of the proof that $\Sigma = \mathcal{Y}.$

The differential ideal $\langle \pi_2, \pi_3 \rangle = \langle \theta_1, \theta_2, \theta_3, \theta_4 \rangle$ descends to $Z \times S^2$ to define a Frobenius distribution. The 4-dimensional integral submanifolds of this distribution are precisely the fibres of the projection $q \colon Z \times S^2 \to X.$ Therefore, the condition $\Sigma = \mathcal{Y}$ implies $\Sigma$ lies entirely in a fiber of $q$. On the other hand, the ideal $\langle \theta_1, \theta_2, \theta_3, \theta_4 \rangle$ descends to $M$ to define a Frobenius distribution whose 3-dimensional integral submanifolds are precisely the canonical leaves, i.e., the fibers of $h \colon M \to X.$ It follows that if $\Sigma$ lies entirely in an $(S^2 \times S^2)$-fiber, then $\Gamma(\Sigma)$ is a subset of a canonical leaf.
\end{proof}

The proof of Theorem \ref{thm:Correspondence}(b) yields more information than is stated in the theorem. In particular, we have:

\begin{cor} \label{cor:MainCorollary}
    The set $D$ consists of the points of $\Sigma$ tangent to the fibers of the projection $q \colon Z \times S^2 \to X.$ Therefore, if $\Sigma \subset Z \times S^2$ is a compact $J$-holomorphic curve nowhere tangent to the fibers of $Z \times S^2 \to X,$ then $\Gamma(\Sigma)$ is a smooth compact Hopf-ruled associative 3-fold in $M.$ 
\end{cor}

\indent Now, recall from $\S$\ref{sec:NKStr} that we have equipped the twistor space $Z$ with its standard strict nearly-K\"{a}hler structure.  Recall also from $\S$\ref{sec:3-Sasakian} that a submanifold of $Z$ is \emph{horizontal} if it is tangent to $\mathsf{H} \subset TZ$, the distribution orthogonal to the fibers of the twistor projection $\tau \colon Z \to X$.  We observe:

\begin{prop}\label{prop:HorzHolo} Let $f \colon \Sigma \to Z$ be a horizontal immersed pseudo-holomorphic curve, equip $\Sigma$ with the induced complex structure, and let $w \colon \Sigma \to S^2$ be a holomorphic map.  Then $(f,w) \colon \Sigma \to Z \times S^2$ is a $J$-holomorphic curve that is nowhere tangent to the fibers of $q \colon Z \times S^2 \to X$. %
\end{prop}

Proposition \ref{prop:HorzHolo} together with Corollary \ref{cor:MainCorollary} will be the key tool in our construction of compact associative submanifolds in the squashed 7-sphere $S^7$ and squashed Aloff-Wallach space $N_{1,1}$ in \S\ref{sec:examples}.

\begin{rmk} Let $(M, \varphi_{a,b}, \mathsf{A})$ be a squashed $3$-Sasakian $7$-manifold. What follows is a heuristic description of the Hopf-ruled associative $3$-folds $\Gamma(\Sigma) \subset M$. Choose a trivialization of $\mathsf{A} \subset TM$, so that the projections $p_w \colon M \to Z$ (for $w \in S^2$) from $\S$\ref{sec:3-Sasakian} are well-defined.  If $\Sigma \subset Z$ is a pseudo-holomorphic curve with respect to the strict nearly-K\"{a}hler structure on $Z$, then for a fixed $w \in S^2$, the pre-image
$$p_w^{-1}(\Sigma) = \bigcup_{z \in \Sigma} p_w^{-1}(z) \subset M$$
is an associative $3$-fold (by the squashed analogue of Proposition \ref{prop:trivial-submanifolds}(a) followed by Proposition \ref{prop:trivial-associatives}) that is clearly Hopf-ruled.  Now, by varying the ruling direction $w$ in a suitable fashion, we can obtain a larger class of examples.  That is, heuristically speaking, the general Hopf-ruled associative in $M$ is a submanifold of the form
$$ \bigcup_{z \in \Sigma} p_{w(z)}^{-1}(\Sigma) \subset M,$$
where the ``directrix surface" $f \colon \Sigma \to Z$ is a pseudo-holomorphic map, and the ``ruling function" $w \colon \Sigma \to S^2$ is a holomorphic map.  As the Correspondence Theorem indicates, this heuristic is valid after removing a discrete set $D$ from $\Sigma$ and avoiding any pair $(f,w) \colon \Sigma \to Z \times S^2$ that lies in an $(S^2 \times S^2)$-fiber of $q \colon Z \times S^2 \to X$.
\end{rmk}

\begin{rmk}
Every $3$-Sasakian $(4n+3)$-manifold $M^{4n+3}$ carries a generalized associative calibration $\phi_K \in \Omega^3(M)$, as defined by Bryant and Harvey \cite{bryharv89}.  We expect that our construction generalizes straightforwardly to that setting.
\end{rmk}

\section{Construction of Compact Associatives}\label{sec:examples}

\indent \indent In this section, we prove Theorems \ref{thm:AssocSphere} and \ref{thm:AssocAW}.  That is, we will apply Proposition \ref{prop:HorzHolo} and Corollary \ref{cor:MainCorollary} to construct infinitely many topological types of non-trivial, compact associative $3$-folds in the squashed $7$-spheres $(S^7, \varphi_{a,b})$ and squashed Aloff-Wallach spaces $(N_{1,1}, \varphi_{a,b})$.

\subsection{The $7$-Sphere} \label{sec:7Sphere}

\indent \indent We begin by considering the $7$-sphere.  There are three natural families of $\G_2$-structures on $S^7$, depending on whether one views the $7$-sphere as
$$\frac{\Spin(7)}{\G_2} \ \ \text{ or } \ \ \frac{\Sp(2)\Sp(1)}{\SO(4)} \ \ \text{ or } \ \ \frac{\Sp(2)}{\Sp(1)}.$$
\indent First, regarding $S^7 = \Spin(7)/\G_2$, it is well-known that $S^7$ admits a unique $\Spin(7)$-invariant $\G_2$-structure $\widehat{\varphi}$ up to scale.  In fact, $\widehat{\varphi}$ is nearly-parallel and induces the round metric on $S^7$.  For the sake of context, in $\S$\ref{sec:StandardS7} we briefly review the literature on associative $3$-folds in $(S^7, \widehat{\varphi})$. \\
\indent Second, regarding $S^7 = \Sp(2)\Sp(1)/\SO(4)$, one can show that $S^7$ admits a $2$-parameter family $\varphi_{a,b}$ of $\Sp(2)\Sp(1)$-invariant $\G_2$-structures, where the parameters are chosen so that $\varphi_{1,1}$ induces the round metric $g_{1,1}$ of constant curvature $1$.  In fact, letting $\mathsf{A} \subset TS^7$ denote the vertical bundle for the quaternionic Hopf fibration $h \colon S^7 \to S^4$, the $\SO(4)$-structure $(\varphi_{1,1}, \mathsf{A})$ is quaternion-Sasakian, and the $2$-parameter family $\varphi_{a,b}$ consists of its squashings.  Consequently, by the discussion in $\S$\ref{sec:Squashed3Sas}, each $\varphi_{a,b}$ is co-closed, and $\varphi_{a,b}$ is nearly-parallel if and only if $b^2 = 5a^2$.  Moreover, $\varphi_{1,1}$ is isometric to the nearly-parallel $\widehat{\varphi}$. \\ 
\indent Third, regarding $S^7 = \Sp(2)/\Sp(1)$, it was recently shown by Loubeau, Moreno, S\`{a} Earp, and Saavedra \cite{loubeau2022harmonic} that the space of $\Sp(2)$-invariant $\G_2$-structures on $S^7$ is parametrized by the $10$-manifold $\R^+ \times \GL_3^+(\R)$.  Note that $\widehat{\varphi}$ and $\varphi_{a,b}$ all belong to this $10$-dimensional family.  We will not consider these $\G_2$-structures here, but refer the reader to \cite{loubeau2022harmonic}.

\subsubsection{Associative Submanifolds in the Round $7$-Sphere $(S^7, \widehat{\varphi})$} \label{sec:StandardS7}

\indent \indent Associative $3$-folds in $(S^7, \widehat{\varphi})$ have been studied in several works:

\begin{itemize}
    \item Fox \cite{fox2007cayley} proved a correspondence theorem relating \emph{geodesically-ruled} associative $3$-folds in $(S^7, \widehat{\varphi})$ to $J$-holomorphic curves in the $12$-manifold $\Gr_2^+(\R^8) \cong \Spin(7)/\U(3)$, where $J$ is a particular $\Spin(7)$-invariant non-integrable almost-complex structure.  Using this, he constructed
    constructed a wealth of geodesically-ruled associative $3$-folds in $(S^7, \widehat{\varphi})$ from minimal surfaces in $S^6 = \Spin(7)/\SU(4)$ and meromorphic functions.  See also \cite[$\S$6]{lotay2012associative}.
    \item Lotay \cite{lotay2012associative} classified the homogeneous associatives in $(S^7, \widehat{\varphi})$ up to $\Spin(7)$-congruence, building upon work of Mashimo \cite{mashimo1985homogeneous}.  Moreover, for each non-constant-curvature minimal $2$-sphere in $S^6$, Lotay constructed a geodesically-ruled associative in $(S^7, \widehat{\varphi})$ that admits an $S^1$-family of non-congruent, global isometric deformations.
    \item Kawai \cite{kawai2014deformations}, \cite{kawai2018second} studied the deformation theory of the $8$ homogeneous associatives in $(S^7, \widehat{\varphi})$.  This was revisited by Moreno and S\`{a} Earp \cite{moreno2017weitzenb} and Moore \cite{moore2019deformations}.
    \end{itemize}
Up to $\Spin(7)$-congruence, the $8$ homogeneous associative $3$-folds in $(S^7, \widehat{\varphi})$ are:
\begin{enumerate}
    \item $A_1 \approx T^3$ is a $T^3$-orbit in $S^7 \subset \C^4$, where $T^3 \leq \SU(4)$ is the maximal torus.  In fact, $A_1$ is special Legendrian: it is the link of the $T^3$-invariant special Lagrangian cone in $\C^4$ discovered by Harvey and Lawson \cite[$\S$III.3]{harvey1982calibrated}.
    \item $A_2 \approx S^3/\Z_3$ is an $\SU(2)$-orbit in $S^7 \subset \C^4$, where the $\SU(2)$-action on $\Sym^3(\C^2) = \C^4$ is irreducible.  In fact, $A_2$ is complex Legendrian: it is of the form $A_2 = p_w^{-1}(\Sigma_0)$ for a suitable $w \in S^2$, where $\Sigma_0 \subset \CP^3$ is the Veronese $\CP^1$, which is horizontal for the twistor fibration $\tau \colon \CP^3 \to S^4$.
    \item $A_3 \approx S^3$ is an $\SU(2)$-orbit in $S^7 \subset \C^4$, where $\SU(2)$ acts on $\Sym^3(\C^2) = \C^4$ via its irreducible representation.
    \item $L_1 \approx S^3$, $L_2 \approx 
    \RP^3$, $L_3 \approx S^3/A_4^*$, and $L_4 \approx S^3/D_3^*$ are Lagrangian submanifolds of a totally geodesic $S^6 \subset S^7$.  In particular, $L_1$ is the link of the Lawson-Osserman coassociatve cone, and is $\U(2)$-invariant.  Also, $L_2 \subset S^5$ is a CR $3$-fold in a totally-geodesic $S^5 \subset S^6$ that is $\SU(2)$-invariant.  Finally, $L_3$ and $L_4$ are orbits of irreducible the $\SO(3)$-action on $\R^7$.
    \item $P_1 \approx S^3$ is the totally-geodesic Hopf $3$-sphere.
\end{enumerate}

\subsubsection{Associative Submanifolds in Squashed $7$-Spheres $(S^7, \varphi_{a,b})$} \label{sec:SquashedS7}

\indent \indent To our knowledge, the only previous work on associative $3$-folds in the squashed $7$-spheres $(S^7, \varphi_{a,b})$ is that of Kawai \cite{kawai2015some}, who classifies the homogeneous associatives in the nearly-parallel $(S^7, \varphi_{a, a\sqrt{5}})$ for $a = 3/5$ and studies their deformation theory. Up to $\Sp(2)\Sp(1)$-congruence, he shows that the $5$ homogeneous associative $3$-folds in $(S^7, \varphi_{a, a\sqrt{5}})$ are:
\begin{enumerate}
    \item $A_1 \approx T^3$ is is a $T^3$-orbit in $S^7 \subset \C^4$, where $T^3 \leq \SU(4)$ is the maximal torus.  In fact, $A_1 = p_w^{-1}(\Sigma)$ for some $w \in S^2$ and some pseudo-holomorphic curve $\Sigma \subset \CP^3$.
    \item $A_2 \approx S^3/\Z_3$ is an $\SU(2)$-orbit in $S^7 \subset \C^4$, where $\SU(2)$ acts on $\Sym^3(\C^2) = \C^4$ via its irreducible representation.  In fact, $A_2 = p_w^{-1}(\Sigma_0)$ for some $w \in S^2$, where $\Sigma_0 \subset \CP^3$ is the Veronese $\CP^1$, which is horizontal for the twistor fibration $\tau \colon \CP^3 \to S^4$. 
    \item $A_3 \approx S^3$ is an $\SU(2)$-orbit in $S^7 \subset \C^4$, where $\SU(2)$ acts on $\Sym^3(\C^2) = \C^4$ via its irreducible representation.  In fact, $A_3 = p_w^{-1}(\Sigma)$ for some $w \in S^2$, where $\Sigma \subset \CP^3$ is a \emph{null-torsion} pseudo-holomorphic $\CP^1$ in $\CP^3$.
    \item $P_1 \approx S^3$ is the totally-geodesic (Hopf) $3$-sphere given by $P_1 = \{(z_1, z_2, 0, 0) \in \C^4\} \cap S^7$.
    \item $P_2 \approx S^3$ is the totally-geodesic $3$-sphere given by $P_2 = \{(z_1, 0, z_3, 0) \in \C^4\} \cap S^7$.
\end{enumerate}
\indent The above associative $3$-folds in $(S^7, \varphi_{a, a\sqrt{5}})$ are all \emph{trivial}, in the sense that they are of the form $p_w^{-1}(\Sigma)$ for a fixed $w \in S^2$ and pseudo-holomorphic curve $\Sigma \subset \CP^3$.  Applying our theory to $M = S^7$, so that $Z = \CP^3$ and $X = S^4$, we can construct non-trivial examples as follows:

\begin{thm} Fix $a,b > 0$.  For every $g \geq 0$, there exists a non-trivial compact associative $3$-fold in $(S^7, \varphi_{a,b})$ diffeomorphic to an $S^1$-bundle over a genus $g$ surface.
\end{thm}

\begin{proof}
\indent Let $\Sigma_g$ be a compact surface of genus $g$.  For each $g \geq 0$, Bryant \cite{bryant1982conformal} has constructed an embedded horizontal, pseudo-holomorphic curve $f \colon \Sigma_g \to \CP^3$.  Equip $\Sigma_g$ with the induced complex structure, and let $w \colon \Sigma_g \to S^2$ be a holomorphic map.  By Proposition \ref{prop:HorzHolo}, the surface $(f,w) \colon \Sigma_g \to Z \times S^2$ is a $J$-holomorphic curve nowhere tangent to an $(S^2 \times S^2)$-fiber.  Thus, for any $a, b >0$, Theorem \ref{thm:Correspondence}(b) implies that $\Gamma(f,w) \subset S^7$ is a $\varphi_{a,b}$-associative $3$-fold diffeomorphic to an $S^1$-bundle over $\Sigma_g$.  If $w \colon \Sigma_g \to S^2$ is chosen to be non-constant, then $\Gamma(f,w) \subset S^7$ is non-trivial.
\end{proof}

\subsection{The Aloff-Wallach Space $N_{1,1}$}

\indent\indent The exceptional Aloff-Wallach space $N_{1,1}$ is the homogeneous space $\left( \SU(3) \times \SO(3) \right)\! / \mathrm{U}(2)$, where the isotropy subgroup $\U(2)$ is the image of the homomorphism $(\rho_1, \rho_2) \colon \mathrm{U}(2) \to \SU(3) \times \SO(3),$ where
\begin{equation*}
        \rho_1 : \mathrm{U}(2) \to \SU(3), \:\:\: A \mapsto \begin{bmatrix}
            \det (A) ^{-1} & 0 \\
            0 & A
        \end{bmatrix},
\end{equation*}
and $\rho_2 \colon \mathrm{U}(2) \to \SO(3) \cong \mathrm{U}(2) / Z(\mathrm{U}(2))$ is the canonical projection.

The space $N_{1,1}$ admits an $(\SU(3) \times \SO(3))$-invariant 3-Sasakian metric, and hence by the results of \S2 and \S3, $N_{1,1}$ admits a two parameter family of co-closed $\G_2$-structures $\varphi_{a,b}$ and a quaternion-Sasakian $\SO(4)$-structure $(\varphi_{1,1}, \mathsf{A})$. The $\G_2$-structure $\varphi_{a,b}$ is nearly parallel if and only if $b^2 = 5 a^2,$ while the $\G_2$-structure $\varphi_{1,1}$ is isometric to a nearly parallel $\G_2$-structure $\widehat{\varphi}$ and induces the normal homogeneous metric on $N_{1,1}$.

Explicitly, we may write the $(\mathfrak{su}(3) \oplus \mathfrak{so}(3))$-valued left-invariant Maurer-Cartan form of $\SU(3) \times \SO(3)$ as
\begin{equation}\label{eq:SUMaurerCartan}
\begin{aligned}
    \sqrt{2} & \begin{bmatrix}
    -i \tfrac{\sqrt{2}}{3} \nu_1 & - \beta_1 + i \beta_2 & -\beta_3 - i \beta_4 \\
    \beta_1 + i \beta_2 & i \tfrac{\sqrt{2}}{6} \nu_1 + i \left( \sqrt{2} \zeta_1 + \tfrac{1}{2} \alpha_1 \right) & \sqrt{2} \left( - \zeta_2 + i \zeta_3 \right) + {\tfrac{1}{2} \left(  -\alpha_2 + i \alpha_3 \right)} \\
    \beta_3 - i \beta_4 & \sqrt{2} \left( \zeta_2 + i \zeta_3 \right) + \tfrac{1}{2} \left( \alpha_2 + i \alpha_3 \right) & i \tfrac{\sqrt{2}}{6} \nu_1 - i \left( \sqrt{2} \zeta_1 + \tfrac{1}{2} \alpha_1 \right)
    \end{bmatrix} \\
    & \oplus \begin{bmatrix}
    0 & 2 \zeta_3 & -2 \zeta_2 \\
    -2 \zeta_3 & 0 & 2 \zeta_1 \\
    2 \zeta_2 & -2 \zeta_1 & 0
    \end{bmatrix}
\end{aligned}
\end{equation}
The vector fields dual to the 1-forms $\nu_1, \zeta_1, \zeta_2, \zeta_3$ generate the isotropy subgroup $\U(2),$ and the 3-form
\begin{equation*}
    \varphi_{1,1} = \alpha_{123} + \alpha_1 \wedge \left(\beta_{12} + \beta_{34} \right) + \alpha_2 \wedge \left(\beta_{13} - \beta_{24} \right) + \alpha_3 \wedge \left(-\beta_{14} - \beta_{23} \right)
\end{equation*}
and distribution
\begin{equation*}
    \mathsf{A} = \ker\!\left( \beta_1, \beta_2, \beta_3, \beta_4 \right)
\end{equation*}
descend to $N_{1,1}$ to define an $(\SU(3) \times \SO(3))$-invariant quaternion-Sasakian $\SO(4)$-structure. The associated 2-parameter family of co-closed $\G_2$-structures is given by
\begin{equation*}
    \varphi_{a,b} = a^3 \, \alpha_{123} + a b^2 \, \left( \alpha_1 \wedge \left(\beta_{12} + \beta_{34} \right) + \alpha_2 \wedge \left(\beta_{13} - \beta_{24} \right) + \alpha_3 \wedge \left(-\beta_{14} - \beta_{23} \right) \right)
\end{equation*}

\subsubsection{Twistor space and canonical fibration}

\indent\indent In the case of the Aloff-Wallach space $N_{1,1}$ the canonical fibration $h \colon M \to X$ is given by the fibration $N_{1,1} \to \mathbb{CP}^2 = \left( \SU(3) \times \SO(3) \right) / \left( \rho_1(\U(2)) \times \SO(3) \right)$. The self-dual Einstein metric $g_X$ on $X$ is simply the Fubini-Study metric on $\mathbb{CP}^2$. In terms of the Maurer-Cartan form (\ref{eq:SUMaurerCartan}) of $\SU(3) \times \SO(3),$ the vector fields dual to the $1$-forms $\nu_1, \zeta_1, \zeta_2, \zeta_3, \alpha_1, \alpha_2, \alpha_3$ generate the $\rho_1(\U(2)) \times \SO(3)$-action, while the $1$-forms $\beta_1, \beta_2, \beta_3, \beta_4$ are semi-basic for the projection from $\SU(3) \to \mathbb{CP}^2.$ The metric $g_X$ is given by
\begin{equation*}
    g_X = \beta_1^2 + \beta_2^2 + \beta_3^2 + \beta_4^2.
\end{equation*}

The twistor space $Z$ of $\mathbb{CP}^2$ is the flag manifold $\SU(3) / \mathrm{T}^2.$ Therefore, the space of Hopf circles in $N_{1,1}$ is diffeomorphic to $\left( \SU(3) / \mathrm{T}^2 \right) \times S^2 = \left( \SU(3) \times \SO(3) \right) / \mathrm{T}^3$. In terms of the Maurer-Cartan form (\ref{eq:SUMaurerCartan}), the $\mathrm{T}^3$-action is generated by the vector fields dual to $\alpha_1,$ $\nu_1,$ and $\zeta_1.$ The almost Hermitian structure $\left( h_{1,1}, J, \omega_{1,1} \right)$ on $\left( \SU(3) / \mathrm{T}^2 \right) \times S^2$ is defined by
\begin{equation*}
    \begin{aligned}
        h_{1,1} &= \alpha_2^2 + \alpha_3^2 + \zeta_2^2 + \zeta_3^2 + 2(\beta_1^2 + \beta_2^2 + \beta_3^2 + \beta_4^2), \\
        \omega &= \alpha_2 \wedge \alpha_3 + \zeta_2 \wedge \zeta_3 + 2(\beta_1 \wedge \beta_2 + \beta_3 \wedge \beta_4).
    \end{aligned}
\end{equation*}

\subsubsection{Pseudo-holomorphic curves in the flag manifold $\mathrm{SU}(3) / \mathrm{T}^2$}\label{sssect:psholflag}

\indent\indent In this subsection, we study pseudoholomorphic curves in the flag manifold $\SU(3) / \mathrm{T}^2$ with the ultimate goal of constructing associative submanifolds of $(N_{1,1}, \varphi_{a,b})$ via Theorem \ref{thm:Correspondence}.

Consider $\C^3$ with the standard hermitian inner product $\langle \cdot, \cdot \rangle.$ Let $Y$ denote the space of complete flags in $\C^3$:
\begin{equation*}
	Y = \left\lbrace \left(\ell, P \right) \mid \ell \in \mathbb{CP}^2 , \: P \in \left( \mathbb{CP}^2 \right)^*, \: \ell \subset P \right\rbrace.
\end{equation*}
Say that a special unitary frame $\left(\mathbf{e}_1, \mathbf{e}_2, \mathbf{e}_3 \right)$ of $\mathbb{C}^3$ is \emph{adapted} to the point $\left(\ell, P \right) \in Y$ if it satisfies the following properties:
\begin{equation*}
	\begin{aligned}
		\ell &= \mathrm{span} \left(\mathbf{e}_1\right), & P &= \mathrm{span} \left(\mathbf{e}_1, \mathbf{e}_2 \right).
	\end{aligned}	
\end{equation*}
The set of adapted coframes may be identified with $\mathrm{SU}(3),$ and this gives an identification of $Y$ with $\SU(3)/\mathrm{T}^2.$ Denote the coset projection $\SU(3) \to \SU(3)/\mathrm{T}^2$ by $\pi$.

There are three $\SU(3)$-equivariant maps $q_i : \SU(3) \to \mathbb{CP}^2,$ each of which factors through a surjection $\tau_i : \SU(3) / \mathrm{T}^2 \to \mathbb{CP}^2$:
\begin{equation*}
	\begin{aligned}
		q_1 : \left(\mathbf{e}_1, \mathbf{e}_2, \mathbf{e}_3 \right) \mapsto \mathrm{span}(\mathbf{e}_1), & & \tau_1 : \left(\ell, P \right) \mapsto \ell, \\
		q_2 : \left(\mathbf{e}_1, \mathbf{e}_2, \mathbf{e}_3 \right) \mapsto \mathrm{span}(\mathbf{e}_2), & & \tau_2 : \left(\ell, P \right) \mapsto P / \ell, \\
		q_3 : \left(\mathbf{e}_1, \mathbf{e}_2, \mathbf{e}_3 \right) \mapsto \mathrm{span}(\mathbf{e}_3), & & \tau_3 : \left(\ell, P \right) \mapsto P^\perp.
	\end{aligned} 
\end{equation*}

Write the (left-invariant) Maurer-Cartan form of $\SU(3)$ in the form
\begin{equation*}
	\gamma = \begin{bmatrix}
		\tfrac{i}{3} \kappa + i \psi & -\overline{\eta_3} & \eta_2 \\
		\eta_3 & \tfrac{i}{3} \kappa - i \psi & -\overline{\eta_1} \\
		-\overline{\eta_2} & \eta_1 & - \tfrac{2i}{3} \kappa
	\end{bmatrix}.
\end{equation*}
Viewing each $\mathbf{e}_i$ as a vector-valued function $\mathbf{e}_i \colon \SU(3) \to \C^3$, we then have the structure equations: 
\begin{equation}\label{eq:SU3framestruct}
	d \begin{bmatrix}
		\mathbf{e}_1 & \mathbf{e}_2 & \mathbf{e}_3 
	\end{bmatrix} = \begin{bmatrix}
	\mathbf{e}_1 & \mathbf{e}_2 & \mathbf{e}_3 
\end{bmatrix} \begin{bmatrix}
		\tfrac{i}{3} \kappa + i \psi & -\overline{\eta_3} & \eta_2 \\
		\eta_3 & \tfrac{i}{3} \kappa - i \psi & -\overline{\eta_1} \\
		-\overline{\eta_2} & \eta_1 & - \tfrac{2i}{3} \kappa
	\end{bmatrix}.
\end{equation}
The Maurer-Cartan equation $d \gamma = - \gamma \wedge \gamma$ on $\SU(3)$ is equivalent to:
\begin{equation}\label{eq:SU3struct}
	\begin{aligned}
		d \eta_1 &= i \left(\kappa - \psi \right) \wedge \eta_1 - \overline{\eta_2 \wedge \eta_3}, & d \kappa &= \tfrac{3i}{2} \left(\eta_1 \wedge \overline{\eta_1} - \eta_2 \wedge \overline{\eta_2}\right), \\
		d \eta_2 &= -i \left( \kappa + \psi \right) \wedge \eta_2 - \overline{\eta_3 \wedge \eta_1}, & d \psi &= \tfrac{i}{2} \left(-\eta_1 \wedge \overline{\eta_1} - \eta_2 \wedge \overline{\eta_2} + 2 \eta_3 \wedge \overline{\eta_3} \right), \\
		d \eta_3 &= 2 i \, \psi \wedge \eta_3 - \overline{\eta_1 \wedge \eta_2}.
	\end{aligned}
\end{equation}
The forms
\begin{equation*}
	\begin{aligned}
		\Omega &= \tfrac{i}{2} \left( \eta_1 \wedge \overline{\eta_1} + \eta_2 \wedge \overline{\eta_2} + \eta_3 \wedge \overline{\eta_3} \right), \:\:\:\: &
		\Upsilon &= \eta_1 \wedge \eta_2 \wedge \eta_3,
	\end{aligned}
\end{equation*}
define the nearly-K\"ahler structure on $\SU(3) / \mathrm{T}^2.$ Let $g_{\mathrm{NK}}$ and $J_{\mathrm{NK}}$ denote the induced Riemannian metric and induced almost complex structure. A 1-form $\theta \in \Omega^1\!\left( \SU(3) / \mathrm{T}^2; \C \right)$ is a $(1,0)$ form if and only if $\pi^* \theta$ lies in the span of $\eta_1, \eta_2, \eta_3.$ Note that while the 1-forms $\eta_i$ do not descend to $Z,$ the subbundles $T_i = \mathrm{span}(\eta_i)$ are well-defined on $\SU(3) / \mathrm{T}^2$ and we have $T^* Z = T_1 \oplus T_2 \oplus T_3$. \\ 

Now suppose that $f\colon \Sigma \to \SU(3) / \mathrm{T}^2$ is a $J_{\mathrm{NK}}$-holomorphic curve. The metric $g_{\mathrm{NK}}$ induces a hermitian structure $h_\Sigma$ on $\Sigma.$ Let $\mathcal{B}(\Sigma) \to \Sigma$ denote the associated $\U(1)$-bundle of unitary coframes on $\Sigma,$ with $\C$-valued tautological 1-form $\sigma \in \Omega^1(\mathcal{B}(\Sigma); \C)$ and Levi-Civita connection form $\phi \in \Omega^1(\mathcal{B}(\Sigma); \mathfrak{u}(1))$. The structure equations of $\mathcal{B}(\Sigma)$ are:
\begin{equation*}
	\begin{aligned}
		d \sigma &= i \phi \wedge \sigma, \:\:\: &		d \phi &= i K \sigma \wedge \overline{\sigma},
	\end{aligned}
\end{equation*}
where $K$ is the Gaussian curvature of $h_\Sigma.$

 Let $\mathcal{F}(\Sigma)$ denote $\mathcal{B}(\Sigma) \times_\Sigma f^* \SU(3)$, a $\mathrm{T}^3$-bundle over $\Sigma.$ The fact that $\Sigma$ is a $J_{\mathrm{NK}}$-holomorphic curve implies by Lemma \ref{lem:holo-char}(b) that there exist $\C$-valued functions $A_1, A_2, A_3$ on $\mathcal{F}(\Sigma)$ such that
 \begin{equation*}
 	\eta_i = A_i \sigma, \:\: i= 1, 2, 3.
 \end{equation*}
The function $A_i$ measures the length of the projection of $T^* \Sigma$ to $T_i.$ Note that
\begin{equation*}
	| A_1 |^2 + |A_2 |^2 + |A_3 |^2 = 1.
\end{equation*}
The structure equations (\ref{eq:SU3struct}) imply that there exist $\C$-valued functions $B_1, B_2, B_3$ on $\mathcal{F}(\Sigma)$ such that
\begin{equation}\label{eq:dAeqns}
	\begin{aligned}
		d A_1 &= B_1 \sigma - i A_1 \phi - i A_1 \psi + i A_1 \kappa, \\
		d A_2 &= B_2 \sigma - i A_2 \phi - i A_2 \psi - i A_2 \kappa, \\
		d A_3 &= B_3 \sigma - i A_3 \phi + 2 i A_3 \psi.
	\end{aligned}
\end{equation}

\begin{prop}
	The cubic form $C = A_1 A_2 A_3 \sigma^3$ descends to $\Sigma$ to define a holomorphic section of the holomorphic line bundle $\mathrm{Sym}^3 \left(T^* \Sigma \right).$
\end{prop}

\begin{proof}
	The equations (\ref{eq:dAeqns}) imply that
	\begin{equation*}
		d( A_1 A_2 A_3) = \left(B_1 A_2 A_3 + B_2 A_3 A_1 + B_3 A_1 A_2 \right) \sigma - 3 i A_1 A_2 A_3 \phi,
	\end{equation*}
	so $A_1 A_2 A_3 \sigma^3$ descends to a well-defined section of  $\mathrm{Sym}^3 \left(T^* \Sigma \right).$ The absence of a $(0,1)$-term in $d(A_1A_2A_3)$ implies this section is holomorphic.
\end{proof}

\begin{defn}
	We shall say a $J$-holomorphic curve $\Sigma \to Z$ is \emph{superholomorphic} if $C$ is identically zero on $\Sigma.$
\end{defn}

\begin{rmk}
	Any rational curve $\mathbb{CP}^1 \to Z$ is superholomorphic, because $\mathbb{CP}^1$ does not support any non-vanishing holomorphic cubic forms.
\end{rmk}

\begin{example}[Lifts of curves in $\mathbb{CP}^2$]\label{eg:liftingcurvesCP}
    The space $\SU(3)$ may also be identified with the Hermitian frame bundle of $\mathbb{CP}^2$. Suppose $f \colon S \to \mathbb{CP}^2$ is a holomorphic curve. We may adapt frames up to a $\mathrm{T}^2$-ambiguity at a point $p \in S$ so that $p = \operatorname{span}(\mathbf{e}_1)$ and $d \mathbf{e}_1 \in \operatorname{span}( \mathbf{e}_1, \mathbf{e}_2 )$ along $S.$ The resulting map $S \to \SU(3) / \mathrm{T}^2$ is known as the \emph{Frenet frame} of $S.$ It is straightfoward to check that each of the three maps $f_i : S \to \SU(3) / \mathrm{T}^2$
    \begin{equation*}
        \begin{aligned}
            & f_1 : \, p \mapsto \left(\mathbf{e}_1 , \operatorname{span}(\mathbf{e}_1, \mathbf{e}_2) \right) \\
           & f_2 : \, p \mapsto \left(\mathbf{e}_2 , \operatorname{span}(\mathbf{e}_2, \mathbf{e}_3) \right) \\
           & f_3 : \, p \mapsto \left(\mathbf{e}_3 , \operatorname{span}(\mathbf{e}_3, \mathbf{e}_1) \right)
        \end{aligned}
    \end{equation*}
    gives a superholomorphic curve $f_i(S) \subset \SU(3) / \mathrm{T}^2$.
$$\begin{tikzcd}
                                            & \mathrm{SU}(3)/\mathrm{T}^2 \arrow[d, "\tau_i"] \\
S \arrow[r, "f"'] \arrow[ru, "f_i", dashed] & \mathbb{CP}^2                                  
\end{tikzcd}$$

\end{example}

\begin{thm}\label{thm:superholom}
    Any superholomorphic curve $\Sigma \subset \SU(3) / \mathrm{T}^2$ arises as the image $f_i(S)$ of a holomorphic curve in $\mathbb{CP}^2$ under one of the three maps $f_i$ defined in Example \ref{eg:liftingcurvesCP}.
\end{thm}

\begin{proof}
    Let $\Sigma \subset \SU(3) / \mathrm{T}^2$ be a superholomorphic curve, so that $A_1A_2A_3$ vanishes on $\mathcal{F}(\Sigma).$ By (\ref{eq:dAeqns}), $A_1,$ $A_2,$ and $A_3$ represent holomorphic sections of bundles over $\Sigma,$ so at least one of them must vanish identically on $\mathcal{F}(\Sigma).$
    
    If $A_3$ vanishes identically on $\mathcal{F}(\Sigma),$ the structure equations (\ref{eq:SU3framestruct}) imply that $\tau_1 (\Sigma)$ is a holomorphic curve in $\mathbb{CP}^2$ and $\mathbf{e}_1, \mathbf{e}_2, \mathbf{e}_3$ is its Frenet frame. It follows that $\Sigma = f_1 \left( \tau_1 (\Sigma) \right).$
 
    The calculations are identical in the cases where $A_1$ or $A_2$ vanish identically on $\mathcal{F}(\Sigma),$ up to a permutation of the $\mathbf{e}_i,$ $f_i,$ and $q_i.$
\end{proof}

\begin{cor}\label{cor:confimmsuperholom}
    Fix $i \in \left\lbrace 1, 2, 3 \right\rbrace.$ Every compact Riemann surface can be conformally immersed in $\SU(3) / \mathrm{T}^2$ as a superholomorphic curve with $A_i = 0.$
\end{cor}

\begin{rmk}
    The cubic form $C$ is essentially the same as the cubic form constructed by Chern--Wolfson \cite{chern1983minimal} and Eells--Wood \cite{eells1983harmonic} in the context of minimal surfaces in $\mathbb{CP}^2.$
\end{rmk}

\subsubsection{Associative submanifolds in squashed $N_{1,1}$}

\indent\indent Applying the results of \S\ref{sssect:psholflag} together with Theorem \ref{thm:Correspondence}(b) we are able to construct non-trivial examples of associative submanifolds in $N_{1,1}.$ Here, an associative 3-fold is considered \emph{trivial} if it is of the form $p_w^{-1}(\Sigma)$ for a fixed projection $p_w : N_{1,1} \to \SU(3) / \mathrm{T}^2$ and pseudo-holomorphic curve $\Sigma \subset \SU(3) / \mathrm{T}^2.$

\begin{thm} Fix $a,b > 0$.  For every $g \geq 0$, there exists a non-trivial compact associative $3$-fold in $(N_{1,1}, \varphi_{a,b})$ diffeomorphic to an $S^1$-bundle over a genus $g$ surface.
\end{thm}

\begin{proof}
    Let $\Sigma_g$ be a compact surface of genus $g$.  For each $g \geq 0$, by Corollary \ref{cor:confimmsuperholom}, there exists an immersed pseudo-holomorphic curve $f \colon \Sigma_g \to \SU(3) / \mathrm{T}^2$ with $A_1 = 0$. By virtue of the condition $A_1 = 0,$ $f$ is horizontal for the twistor projection $\SU(3) / \mathrm{T}^2 \to \mathbb{CP}^2.$ Equip $\Sigma_g$ with the induced complex structure, and let $w \colon \Sigma_g \to S^2$ be a holomorphic map.  By Proposition \ref{prop:HorzHolo}, the surface $(f,w) \colon \Sigma_g \to \left( \SU(3) / \mathrm{T}^2 \right) \times S^2$ is a $J$-holomorphic curve nowhere tangent to an $(S^2 \times S^2)$-fiber.  Thus, for any $a, b >0$, Theorem \ref{thm:Correspondence}(b) implies that $\Gamma(f,w) \subset N_{1,1}$ is a $\varphi_{a,b}$-associative $3$-fold diffeomorphic to an $S^1$-bundle over $\Sigma_g$.  If $w \colon \Sigma_g \to S^2$ is chosen to be non-constant, then $\Gamma(f,w) \subset N_{1,1}$ is non-trivial.
\end{proof}

\bibliographystyle{plain}
\bibliography{FBRef}

\Addresses

\end{document}